\documentclass[10pt,a4paper]{amsart} 

\usepackage[utf8]{inputenc}
\usepackage[T1]{fontenc}
\usepackage{amsmath, amssymb}
\usepackage{amsthm}
\usepackage{mathtools}   

\usepackage{xspace} 
\usepackage{xfrac}  
\usepackage{xifthen} 
\usepackage{enumitem} 

\usepackage{hyperref}
\hypersetup{
    unicode=true,          
    pdftitle={Discretized Sum-Product Estimates in Matrix Algebras},    
    pdfauthor={He Weikun},     
    pdfsubject={Arithmetic combinatorics},   
    pdfkeywords={Discretized}{Sum-product}, 
    pdfborder={0 0 0},
    colorlinks=true,       
    linkcolor=black,          
    citecolor=black,        
    filecolor=black,      
    urlcolor=black,           
    pagebackref=true
}

\setenumerate[0]{label=(\roman*)}   

\renewcommand{\phi}{\varphi} 

\newcommand\dash{\nobreakdash-\hspace{0pt}}

\newtheorem{thm}{Theorem}
\newtheorem{coro}[thm]{Corollary}
\newtheorem{prop}[thm]{Proposition}
\newtheorem{lemm}[thm]{Lemma}

\theoremstyle{definition}

\newtheorem*{déf}{Definition}
\newtheorem*{rmq}{Remark}

\def\N{{\mathbb N}}
\def\Z{{\mathbb Z}}
\def\R{{\mathbb R}}
\def\C{{\mathbb C}}

\def\dd{{\,\mathrm{d}}}
\def\HH{{\mathbb H}}

\def\CC{{\mathcal C}}  
\def\UC{{\mathcal U}}  

\newcommand{\gf}{\mathfrak{g}} 

\def\OB{{\mathbf O}}  

\DeclareMathOperator{\End}{End}
\DeclareMathOperator{\Mat}{\mathcal{M}}

\DeclareMathOperator{\SL}{SL}
\DeclareMathOperator{\GL}{GL}

\DeclareMathOperator{\SO}{SO}
\DeclareMathOperator{\SU}{SU}

\DeclareMathOperator{\Id}{Id}

\DeclareMathOperator{\Ball}{\mathbf{B}}

\DeclareMathOperator{\Gr}{Gr}        

\DeclareMathOperator{\Span}{Span}
\DeclareMathOperator{\Supp}{Supp}
\DeclareMathOperator{\tr}{tr}

\DeclareMathOperator{\En}{\omega}        
\DeclareMathOperator{\ad}{ad}            
\DeclareMathOperator{\diag}{diag}

\newcommand{\transp}[1]{\prescript{t}{}{#1}} 
\newcommand{\abs}[1]{\lvert#1\rvert}    
\newcommand{\norm}[1]{\lVert#1\rVert}   

\newcommand{\ceil}[1]{\left\lceil#1\right\rceil}   
\newcommand{\sg}[1]{\left<#1\right>}               

\newcommand{\inv}[1]{#1^{-1}}                  
\newcommand{\Inv}[1]{\frac{1}{#1}}               

\newcommand{\ens}[1]{\left\lbrace#1\right\rbrace}     
\newcommand{\ensA}[1]{\{1,\dotsc,#1\}}     

\newcommand{\Prob}[2][]{\ifthenelse{\equal{#1}{}}   
            {\mathbb{P}}               
            {\mathbb{P}_{#1}}\bigl[#2\bigr]}

\newcommand{\Espr}[2][]{\ifthenelse{\equal{#1}{}}   
            {\mathbb{E}}               
            {\mathbb{E}_{#1}}\bigl[#2\bigr]}

\newcommand{\muex}[1][]{\ifthenelse{\isempty{#1}}   
            {\mu_{\mathrm{ex}}}               
            {\mu_{\mathrm{ex}}^{#1}}}

\newcommand{\smatrixp}[1]{\bigl( \begin{smallmatrix} #1 \end{smallmatrix} \bigr)} 

\newcommand{\Ncov}[1][\delta]{\mathcal{N}_{#1}}
\newcommand{\Sset}[1][\delta]{S_{#1}}

\renewcommand{\bullet}{\boldsymbol{\,\cdot\,}}

\title{Discretized sum-product estimates in matrix algebras}

\author{Weikun He}
\address{Laboratoire de Mathématiques d'Orsay, Univ. Paris-Sud, Université Paris-Saclay, 91405 Orsay, France.}
\email{weikun.he@math.u-psud.fr}

\begin{document}

\begin{abstract}
We generalize Bourgain's discretized sum-product theorem to matrix algebras.
\end{abstract}

\maketitle

\section{Introduction}

Let $E$ be a ring. For subsets $A$ and $B$ of $E$, write
\[-A = \ens{-a\mid a \in A},\]
\[A + B = \ens{a + b \mid a \in A, b \in B}\]
and
\[A\cdot B = \ens{ab \mid a \in A, b \in B}.\]
Let $s$ be a positive integer, we write $sA$ for the $s$-fold sum-set $A + \dotsb + A$ and $A^s$ for the $s$-fold product-set $A \dotsm A$. Moreover, we define recursively $\sg{A}_1 = A\cup(-A)$ and for all positive integers $s$,
\[\sg{A}_{s+1} = \sg{A}_{s}\cup \bigl(\sg{A}_{s} + \sg{A}_1\bigr) \cup \bigl(\sg{A}_{s}\cdot \sg{A}_1\bigr).\]
In other words, $\sg{A}_{s}$ is the set of elements obtained from at most $s$ elements of $A\cup(-A)$ by adding and multiplying them.

Roughly speaking, the sum-product problem asks, given a set $A$, whether $A$ grows fast under addition, multiplication and subtraction and if not, what are the obstructions. Thus, a sum-product estimate is a lower bound for the size of $\sg{A}_s$ with $s \geq 2$. The first result of this type is due to Erd\H{o}s and Szemerédi~\cite{ErdosSzemeredi} for the ring of reals $\R$. Since then, numerous works followed, either establishing sum-product estimates for broader classes of rings or improving existing bounds. See~\cite{Tao_Ring} for an elegant treatment and more history. In most of these estimates, the size of a set is measured by its cardinality. 
\subsection{Bourgain's discretized sum-product theorem} 

In \cite{KatzTao}, Katz and Tao introduced the $\delta$-discretized setting where the size of a set is measured by its covering number by $\delta$-balls. More precisely, assume that the ring $E$ comes with a distance. We denote by $\Ball(x,\rho)$ the closed ball with center $x \in E$ and radius $\rho > 0$. Let $\delta > 0$ be a positive real number, which will be referred to as the scale. For a bounded subset $A$ of $E$, its covering number $\Ncov(A)$ is the minimal number of points $x_1,\dotsc,x_N$ such that closed balls $\Ball(x_1,\delta), \dotsc, \Ball(x_N,\delta)$ cover $A$. Katz and Tao conjectured a sum-product estimate in this discretized setting for the ring $\R$. The conjecture was solved by Bourgain in~\cite{Bourgain2003,Bourgain2010} :

\begin{thm}[Bourgain~\cite{Bourgain2010}]\label{thm:Bourgain}
Given $\kappa > 0$ and $\sigma < 1$, there is $\epsilon > 0$ such that the following holds for $\delta > 0$ sufficiently small. Let $A$ be a subset of $\R$, assume that
\begin{enumerate}
\item $A \subset \Ball(0,\delta^{-\epsilon})$,
\item $\forall \rho \geq \delta$, $\Ncov[\rho](A) \geq \delta^\epsilon \rho^{-\kappa}$,
\item $\Ncov(A) \leq \delta^{-\sigma - \epsilon}$.
\end{enumerate}
Then,
\begin{equation}\label{eq:sumNproduct}
\Ncov(A + A) + \Ncov(A\cdot A) \geq \delta^{-\epsilon} \Ncov(A).
\end{equation}
\end{thm}

This theorem is of very different nature from its discrete analogue~\cite{ErdosSzemeredi}. Its proof required an involved multi-scale analysis. Additive structure is analyzed at each level (using a quantitative version of Freiman's theorem due to Chang~\cite{Chang2002} in Bourgain's first proof in~\cite{Bourgain2003}) and then multiplicative information is used to show growth in size.

When the discretized setting was introduced in~\cite{KatzTao}, it was conceived as a general strategy for proving results in the continuous regime (where sets are measured by its Hausdorff dimension) while using ideas and results in the discrete regime where we have well-developed theory such as arithmetic combinatorics. Thus, one of the original motivations of Bourgain's theorem was the Erd\H{o}s-Volkmann ring conjecture~\cite{ErdosVolkmann} which asserts that no Borel measurable subring of $\R$ has Hausdorff dimension strictly between $0$ and $1$. This conjecture was also settled by Edgar and Miller~\cite{EdgarMiller} using different ideas. However, their proof is not quantitative. Moreover, as anticipated in~\cite{KatzTao}, Bourgain's theorem also makes progress in geometric measure theory, namely on the Falconer distance problem and the Furstenberg conjecture.

It turns out that the influence of Bourgain's theorem does not stop here. More deep works have been done using this theorem. For example, it allowed Bourgain and Yehudayoff~\cite{BourgainYehud} to construct explicit monotone expanders. And, in~\cite{BourgainGamburd_SU2,BourgainGamburd_SU}, Bourgain and Gamburd showed a spectral gap result in $\SU(d), d\geq 2$. Benoist and Saxcé~\cite{BenoistSaxce} extended this work to all compact simple Lie groups. Introducing the notion of local spectral gap, Boutonnet-Ioana-Salehi Golsefidy~\cite{BISG} further generalized the result to non-compact settings. Both of these generalizations build on Saxcé's discretized product theorem~\cite{Saxce}, whose proof uses Bourgain's theorem in a crucial way. Let us also mention that combining Saxcé's product theorem with Fourier analysis on Lie groups, Lindenstrauss-Saxcé~\cite{LindenSaxce} and Saxcé~\cite{Saxce_subgroup} obtained Erd\H{o}s-Volkmann type results for Borel measurable subgroups in simple Lie groups.

In the same paper~\cite{Bourgain2010}, Bourgain deduced from his sum-product theorem a Marstrand type projection theorem in the discretized setting. This result is one of the main ingredients in the Bourgain-Furman-Lindenstrauss-Mozes theorem~\cite{BFLM} on the equidistribution for orbits of subsemigroups of $\SL_d(\Z)$ on the torus. This last result can be used to study stationary measures. Namely, they obtained a stiffness result. It's worth noting that the Bourgain-Furman-Lindenstrauss-Mozes theorem was applied by Bourgain and Varj\'u in~\cite{BourgainVarju} to show expansion in $\SL_d(\Z/q\Z)$ with $q$ an arbitrary integer.

\subsection{Statement of the main results}

The aim of this article is to generalize Bourgain's discretized sum-product theorem to matrix algebras. Let $E$ be a real algebra of finite dimension endowed with a norm that makes it a normed vector space. Such a structure will be called a normed algebra. We want to understand, given a bounded subset $A \subset E$ satisfying similar properties as in Theorem~\ref{thm:Bourgain}, how $\Ncov(\sg{A}_s)$ grows and whether \eqref{eq:sumNproduct} or similar estimates hold.

If we ask these questions for general real algebras, they can be as hard as the Freiman problem\footnote{Freiman problem asks, in a given abelian group, which subsets grow slowly under addition. Freiman's Theorem asserts that they are "close" to generalized arithmetic progressions. Obtaining a polynomial bound for this theorem is one of the fundamental open problems in additive combinatorics. See~\cite[Chapter 5]{TaoVu}.} as illustrated by the following example. Let $A_0$ be a bounded subset of $\R$ containing $0$. Consider $A$ the set of matrices of the form $\smatrixp{1 & a\\0 & 1}$ with $a \in A_0$. Then for any positive integer $s$, every element in $\sg{A}_{s}$ is of the form $\smatrixp{k & a\\0 & k}$ with $a \in s'A_0$ and $k \in \ens{-s', \dotsc, s'}$ where $s'$ is an integer depending on $s$. And conversely, for every $a \in sA_0$, we have $\smatrixp{s & a\\0 & s} \in \sg{A}_{s}$. Hence
\[\Ncov(sA_0) \leq \Ncov(\sg{A}_{s}) \ll_s \Ncov(s'A_0).\]
This means the growth of $A$ under addition and multiplication is to some extent equivalent to the growth of $A_0$ under only addition.

For this reason, we will restrict our attention to simple algebras. Recall that by the Wedderburn structure theorem and the Frobenius theorem, a simple real algebra of finite dimension is isomorphic to $\Mat_n(\R)$, $\Mat_n(\C)$ or $\Mat_n(\HH)$, the algebra of $n \times n$ matrices over the real numbers, the complex numbers, or the quaternions, for some $n \geq 1$. 

\begin{thm}\label{thm:SumProduct2}
Let $E$ be a normed simple real algebra of finite dimension. Given $\kappa > 0$ and $\sigma < \dim(E)$, there is $\epsilon > 0$ depending on $E$, $\kappa$ and $\sigma$ such that the following holds for $\delta > 0$ sufficiently small. Let $A$ be a subset of $E$, assume that
\begin{enumerate}
\item $A \subset \Ball(0,\delta^{-\epsilon})$,
\item \label{it:nonCon} $\forall \rho \geq \delta$, $\Ncov[\rho](A) \geq \delta^\epsilon \rho^{-\kappa}$,
\item \label{it:Alarge} $\Ncov(A) \leq \delta^{-\sigma - \epsilon}$,
\item \label{it:awayAlg} for every proper subalgebra $W \subset E$, there is $a \in A$ such that $d(a,W) \geq \delta^\epsilon$. 
\end{enumerate}
Then,
\begin{equation}\label{eq:SumProduct2}
\Ncov(A + A) + \Ncov(A + A\cdot A) \geq \delta^{-\epsilon} \Ncov(A).
\end{equation}
\end{thm}

The case $E = \C$ is due to Bourgain and Gamburd~\cite{BourgainGamburd_SU2}. So our result is new for $\dim(E) > 2$. Let us also mention the work of Chang~\cite{Chang2007} who investigated the sum-product problem for real matrices and the result of Tao in~\cite{Tao_Ring} concerning general algebras, both in the discrete context.
 
The assumption \ref{it:nonCon} is a non-concentration condition. It is to avoid the situation where $A$ is a union of a bounded number of small balls. A subset $A\subset E$ satisfying the condition \ref{it:awayAlg} will be said to be $\delta^\epsilon$-away from subalgebras. This is the additional condition compared to the one-dimensional case. Without it $A$ can be trapped in a small neighborhood of a proper subalgebra. Note the conclusion~\eqref{eq:SumProduct2} is slightly weaker than that of Theorem~\ref{thm:Bourgain}. Here, instead of $A\cdot A$, we need $A + A\cdot A$ to see the growth. Actually, the estimate~\eqref{eq:sumNproduct} fails under the same assumptions as soon as $\dim(E)$ is greater than $1$. Indeed, if $A$ is a union of a segment of unit length and an orthonormal basis of $E$, then the set $A$ satisfies the assumptions\footnote{The condition $\dim(E) > 1$ is needed to have assumption \ref{it:Alarge}.} of Theorem~\ref{thm:SumProduct2} but $A + A$ and $A\cdot A$ are both unions of a bounded number of unit segments. Thus \eqref{eq:sumNproduct} fails for such $A$.

Our second result concerns linear actions on Euclidean spaces. Let $X$ be a bounded subset of the Euclidean space $\R^n$. Let $A \subset \End(\R^n)$ be a collection of linear endomorphisms. We can ask whether $X$ grows under addition and transformation by elements of $A$, provided that $A$ is sufficiently rich.

\begin{thm}\label{thm:ActionRn}
Let $n$ be a positive integer. Given $\kappa > 0$ and $\sigma < n$, there is $\epsilon > 0$ such that the following holds for $\delta > 0$ sufficiently small. Let $A$ be a subset of $\End(\R^n)$ and $X$ a subset of $\R^n$, assume that
\begin{enumerate}
\item $A \subset \Ball(0,\delta^{-\epsilon})$,
\item $\forall \rho \geq \delta$, $\Ncov[\rho](A) \geq \delta^\epsilon \rho^{-\kappa}$,
\item \label{it:irreducible}for every proper nonzero linear subspace $W \subset \R^n$, there is $a \in A$ and $w \in W \cap \Ball(0,1)$ such that $d(aw,W) \geq \delta^\epsilon$.
\item $X \subset \Ball(0,\delta^{-\epsilon})$,
\item $\forall \rho \geq \delta$, $\Ncov[\rho](X) \geq \delta^\epsilon \rho^{-\kappa}$,
\item $\Ncov(X) \leq \delta^{-\sigma - \epsilon}$.
\end{enumerate}
Then,
\begin{equation}\label{eq:ActionRn}
\Ncov(X + X) + \max_{a \in A}\Ncov(X+aX) \geq \delta^{-\epsilon}\Ncov(X),
\end{equation}
where $aX = \ens{ax\mid x \in X}$.
\end{thm}

This improves a previous result of Bourgain and Gamburd~\cite[Proposition 1]{BourgainGamburd_SU} where a constant is required instead of $\delta^\epsilon$ in the irreducibility condition \ref{it:irreducible}. The proof of Bourgain and Gamburd seems to rely on this irreducibility hypothesis at all scales in a crucial way. Relaxing this hypothesis is the most important technical challenge in the proof of Theorem~\ref{thm:ActionRn}. A reason for which this improvement is important is that this kind of estimates are often used together with the Balog-Szemerédi-Gowers theorem, which requires restricting the sets we work with to subsets of size $\delta^\epsilon$ times the original size. This usually destroys all information above the scale $\delta^\epsilon$.

As a simple corollary, we can obtain a "sum-bracket" estimate in simple Lie algebras. If $A$ is a subset of a Lie algebra $\gf$, write $[A,A] = \ens{ [a,b] \mid a,b \in A}$.
\begin{coro}\label{cr:SumBracket}
Let $\gf$ be a normed\footnote{We mean a norm which makes the underlying linear structure a normed vector space.} simple Lie algebra of finite dimension. Given $\kappa > 0$ and $\sigma < \dim(\gf)$, there is $\epsilon > 0$ such that the following holds for $\delta > 0$ sufficiently small. Let $A$ be a subset of $\gf$, assume that
\begin{enumerate}
\item $A \subset \Ball(0,\delta^{-\epsilon})$,
\item $\forall \rho \geq \delta$, $\Ncov[\rho](A) \geq \delta^\epsilon \rho^{-\kappa}$,
\item $\Ncov(A) \leq \delta^{-\sigma - \epsilon}$,
\item for every proper Lie subalgebra $W$ of $\gf$, there is $a \in A$ such that $d(a,W) \geq \delta^\epsilon$.
\end{enumerate}
Then, 
\[\Ncov(A + A) + \Ncov(A + [A,A]) \geq \delta^{-\epsilon} \Ncov(A).\]
\end{coro}

\subsection{Motivation}
Theorem~\ref{thm:ActionRn} is an interesting result on its own right, and we hope it leads to fruitful applications, just like the one dimensional case did. In particular, our work is primarily motivated by the following specific application. In~\cite{BenoistQuint}, Benoist and Quint have generalized the stiffness result of Bourgain-Furman-Lindenstrauss-Mozes~\cite{BFLM} to a much broader class of dynamical systems. In particular, for linear actions on tori, they do not need the proximality assumption in~\cite{BFLM}. However, the results in~\cite{BenoistQuint} are not quantitative. 

The approach in~\cite{BFLM} is Fourier-analytic. While a subgroup $\Gamma \subset \SL_d(\Z)$ acts on the torus, its transpose $\transp{\Gamma}$ acts on Fourier coefficients. A large part of the proof in~\cite{BFLM} focuses on the study of large Fourier coefficients under this action. By the theory of random matrix products, if $\Gamma$ is proximal, then large random products in $\Gamma$ behave like rank one projections composed with rotations, if viewed at an appropriate scale. That is how Bourgain's discretized projection theorem comes into play. It is clear that if one wants to avoid the proximality assumption in the Bourgain-Furman-Lindenstrauss-Mozes theorem, a higher rank discretized projection theorem is required. And as the rank one projection theorem follows from the sum-product theorem in $\R$, a higher rank projection theorem can be proved using Theorem~\ref{thm:ActionRn}. This is the subject of a subsequent paper~\cite{He_proj}.

\subsection{Outline of the proofs}
Both Theorem~\ref{thm:SumProduct2} and Theorem~\ref{thm:ActionRn} are deduced from the following theorem. 

\begin{thm}\label{thm:SumProduct}
Let $E$ be a normed simple real algebra of finite dimension. Given $\kappa > 0$ and $\epsilon_0 > 0$, there is $\epsilon > 0$ and an integer $s \geq 1$ such that the following holds for $\delta > 0$ sufficiently small. Let $A$ be a subset of $E$, assume that
\begin{enumerate}
\item $A \subset \Ball(0,\delta^{-\epsilon})$,
\item $\forall \rho \geq \delta$, $\Ncov[\rho](A) \geq \delta^\epsilon \rho^{-\kappa}$,
\item $A$ is $\delta^\epsilon$-away from subalgebras.
\end{enumerate}
Then, 
\[\Ball(0,\delta^{\epsilon_0}) \subset \sg{A}_{s} + \Ball(0,\delta).\]
\end{thm}

Note that each of the conditions (ii) and (iii) rules out one obvious obstruction for $\sg{A}_s$ to grow. Indeed, firstly, if $A$ is covered by a bounded number of balls of radius $\rho$ with $\rho < \delta^{\epsilon_0}$, then $\sg{A}_s$ is covered by $O_s(1)$ balls of radius $\rho$. Secondly, if $A$ is contained in the unit ball\footnote{In this example, the norm on $E$ is submultiplicative, i.e. $\forall x,y \in E,\; \norm{xy}\leq \norm{x}\norm{y}$. This assumption is not restrictive since every norm on $E$ is equivalent to a submultiplicative one.} and in the $\rho$-neighborhood of a proper subalgebra with $\rho < \delta^{\epsilon_0}$, then $\sg{A}_s$ is contained in the $O_s(\rho)$-neighborhood of the same proper subalgebra.

The main ingredient in the proof of Theorem~\ref{thm:SumProduct} is a sum-product theorem~\cite[Corollary 8]{BourgainGamburd_SU} due to Bourgain-Gamburd concerning the ring $\C^n$, the $n$-fold direct product of $\C$ with itself. Let $n$ be a positive integer. We denote by $\Delta$ the set of diagonal matrices in $\Mat_n(\C)$.
\begin{thm}[Bourgain-Gamburd~\cite{BourgainGamburd_SU}]\label{thm:etaDelta}
Given $\kappa > 0$ and $n$ a positive integer, there is a positive integer $s \geq 1$ such that, for $\delta > 0$ sufficiently small, the following holds. Let $A$ be a subset of $\Mat_n(\C)$. Assume that
\begin{enumerate}
\item $A \subset \Ball(0,1)$,
\item $\Ncov(A) \geq \delta^{-\kappa}$,
\item $A \subset \Delta + \Ball(0,\delta)$.
\end{enumerate}
Then there is $\eta \in \Delta$ with $\norm{\eta} = 1$ such that
\begin{equation*}
[0,\delta^\alpha] \eta \subset \sg{A}_s + \Ball(0,\delta^{\alpha + \beta}),
\end{equation*}
with some $0 \leq \alpha < C(n,\kappa)$ and some $\beta > c(n,\kappa) > 0$. 
\end{thm}

Let us sketch the proof of Theorem~\ref{thm:SumProduct}. In the following paragraphs, each $s$ stands for some unspecified integer that can be bounded in terms of $E$ and $\kappa$. In order to use the Bourgain-Gamburd theorem above, we need first to embed the algebra $E$ in $\Mat_n(\C)$ and then produce a lot of nearly simultaneously diagonalizable elements in $\sg{A}_s$. The standard way (since the work of Helfgott~\cite{Helfgott08}) to produce such elements is to use the fact that the centralizer of a matrix with $n$ distinct eigenvalues is simultaneously diagonalizable. Since the set of matrices with at least one multiple eigenvalue is an algebraic subvariety of $\Mat_n(\C)$, to find an element $a \in \sg{A}_s$ with $n$ distinct eigenvalues we use the technique of "escape from subvarieties", first developed in~\cite{EskinMozesOh}. For our discretized setting, a quantitative version of this technique is required since distance matters. For Lie groups, this is established in~\cite{Saxce}. Here we adapt the argument in the sum-product setting.

Once we have such an element $a$, we consider the map $\phi \colon x \mapsto ax - xa$. We distinguish two cases. 
\begin{enumerate}[label=(\alph*)]
\item If $\phi(A)$ is large ($\Ncov(\phi(A)) \geq \delta^{\kappa'}\Ncov(A)$ with $\kappa' = \frac{\kappa}{3\dim(E)}$), then we will prove $\Ncov(\sg{A}_s) \geq \delta^{-\kappa'}\Ncov(A)$ in this case. We remark that all element in $\phi(A)$ have zero trace. Hence if $B$ is a set of matrices with a lot of different traces, then $\phi(A) + B$ contains a lot of disjoint translates of $\phi(A)$. In particular, $\Ncov(\phi(A) + B) \gg \Ncov(\phi(A))\Ncov(\tr(B))$. Thus, it suffices to establish a lower bound on the size of the set of traces of $\sg{A}_s$. Indeed, we can prove $\tr(\sg{A}_s) \geq \delta^{-2\kappa'}$ using the fact that the bilinear form $(x,y)\mapsto \tr(xy)$ is non-degenerate.
\item Otherwise the set $A$ must have a large intersection with a fiber of $\phi$, i.e. there is $y \in \Mat_n(\C)$ such that $A \cap \inv{\phi}(\Ball(y,\delta)) \geq \delta^{-\kappa'}$. The difference set of the above intersection consists of nearly simultaneously diagonalizable matrices. And we can apply Theorem~\ref{thm:etaDelta} to get a small segment at a smaller scale (the segment is inside the $\delta^{\alpha+\beta}$-neighborhood of $\sg{A}_s$).
\end{enumerate}
What we do is to repeat the same argument to $\sg{A}_s$ if case (a) happens. After a bounded number of times, case (a) won't be possible because $\sg{A}_s \subset \Ball(0,O_s(\delta^{-O_s(\epsilon)}))$. Hence eventually, case (b) is true, i.e. inside $\sg{A}_s$, there is a segment of direction $\xi$ and length $\delta^{\alpha}$ at scale $\delta^{\alpha + \beta}$. Then, using the fact that the two-sided ideal generated by $\xi$ is the whole algebra $E$, we can prove that the small segment will grow into a small ball under left and right multiplication by elements of $A$. 

This almost finishes the proof. The only problem is that the ball obtained is not large enough and it is at a scale other than $\delta$. As in~\cite{Saxce}, this issue can be solved by applying the above argument at various scales ranging from $\delta^{\Inv{\alpha+\beta}}$ to $\delta^{\frac{\epsilon_0}{\alpha}}$.

That is how the proof of Theorem~\ref{thm:SumProduct} goes. To deduce Theorem~\ref{thm:SumProduct2} from it, we argue by contradiction and use the fact that if both $A + A$ and $A + A\cdot A$ are small (i.e. \eqref{eq:SumProduct2} fails), then for every $s$, $\sg{A}_s$ is small, and thus cannot grow into a large ball as Theorem~\ref{thm:SumProduct} asserts.

To prove Theorem~\ref{thm:ActionRn}, a little more work is needed. First, in the special case where the collection of endomorphisms is so large that for every $x \in X$, $Ax = \ens{ax \mid a \in A}$ contains a ball of radius $\norm{x}$, a Fubini-type argument yields \eqref{eq:ActionRn}. Then, using additive combinatorics, we can show that if \eqref{eq:ActionRn} fails, then we have an upper bound for $\Ncov(X + aX)$ for every $a \in \sg{A}_s$, $s \geq 1$. Therefore, the idea of the proof is to apply Theorem~\ref{thm:SumProduct} to make $A$ grow into a fat ball in some subalgebra $E \subset \Mat_n(\R)$ so that we can use the special case. Here the subalgebra $E$ can be understood as the subalgebra approximately generated by the set $A$. It inherits the irreducibility property (assumption~\ref{it:irreducible} in Theorem~\ref{thm:ActionRn}) from $A$. In particular, $\R^n$ is an irreducible representation of $E$. Hence, by the Wedderburn structure theorem, $E$ is isomorphic to $\Mat_n(\R)$ or $\Mat_{\frac{n}{2}}(\C)$ or $\Mat_{\frac{n}{4}}(\HH)$. Here, a technical issue appears : in Theorem~\ref{thm:SumProduct}, the result depends on the norm on $E$. In the present situation, the norm on $E$ is induced from that on $\Mat_n(\R)$. To have a control on it, we need a quantitative version of the Wedderburn theorem. Indeed, we show that under the quantitative irreducibility condition~\ref{it:irreducible} of Theorem~\ref{thm:ActionRn}, the normed algebra $E$ is isomorphic to one of the three matrix algebras endowed with standard operator norm via a bi-Lipschitz map with the Lipschitz constant controlled independently of $A$.

\subsection{Organization of the paper}

In Section~\ref{sc:prelim} we introduce some definitions and notations and then recall some useful tools from additive combinatorics and the theory of semianalytic sets. Sections~\ref{sc:escape}--\ref{sc:Wedderburn} prepare for the proof of the main results. More precisely, Section~\ref{sc:escape} is dedicated to the "escape from subvariety" technique. Section~\ref{sc:trace} deals with a lower bound on the size the set of traces. And in Section~\ref{sc:Wedderburn} we establish an effective version of the Wedderburn structure theorem. We complete the proof of Theorem~\ref{thm:SumProduct} and deduce Theorem~\ref{thm:SumProduct2} in Section~\ref{sc:SumProduct}. Finally Theorem~\ref{thm:ActionRn} is proved in Section~\ref{sc:ActionRn} and Corollary~\ref{cr:SumBracket} is deduced in Section~\ref{sc:Lie}.

\subsection*{Acknowledgements}
This work is part of my Ph.D. thesis. I would like to thank my advisors Emmanuel Breuillard and Péter Varj\'u for guiding me during my research. I am also grateful to Nicolas de Saxcé for very helpful discussions. I thank all of them as well as the anonymous referee for their invaluable advice which improved greatly the presentation of this paper.

\section{Preliminaries}\label{sc:prelim}
We first set up notation and terminology and then recall some tools that we shall need such as the Ruzsa calculus and the \L{}ojasiewicz inequality.

\subsection{Notations and definitions} 
Throughout this paper, $n$ denotes a positive integer. We use Landau notations $f = O(g)$ and Vinogradov notations $f \ll g$. Most of our estimates are about objects in some ambient space (a normed vector space or a normed algebra) and we write $f \ll_V g$ and $f = O_V(g)$ to indicate that the implied constant depends not only on the dimension of $V$ but also on the norm of $V$. And we omit the subscript when it depends only on the dimension $n$.

We endow the space $\R^n$ with its usual Euclidean norm $\norm{\bullet}$ and $\C^n$ and $\HH^n$ with their respective $l^2$-norm. All algebras are over $\R$ and unitary. In an algebra $E$, $1_E$ denotes the multiplicative identity. All subalgebras of $E$ contain $1_E$. When $K$ is a division algebra over $\R$, denote by $\Mat_n(K)$ the algebra of $n$ by $n$ matrices with coefficients in $K$. For a real linear space $V$, denote by $\End(V)$ the algebra of real endomorphisms of $V$. We identify $\End(\R^n)$ with $\Mat_n(\R)$ in the usual way. For $0 \leq k \leq \dim(V)$, denote by $\Gr(k,V)$ the Grassmannian of $k$-dimensional subspaces in $V$. In any normed vector space, $d(\bullet,\bullet)$ stands for the distance induced by the norm and $\Ball(x,r)$ stands for the closed ball with center $x$ and radius $r$.

Let $\delta > 0$. As mentioned in the introduction, if $A$ is a bounded subset of a normed vector space $V$, we denote by $\Ncov(A)$ its \emph{external covering number} by $\delta$-balls. This number is also known as the \emph{metric entropy} of $A$. We say $A$ is $\delta$-separated if for any $a \in A$, $a$ is the only element in the intersection $\Ball(a,\delta)\cap A$. The following properties are used throughout this paper. If $\tilde A$ is a maximal $\delta$-separated subset of $A$, then
\[\Ncov(A) \leq \abs{\tilde A} \leq \Ncov[\frac{\delta}{2}](A) \ll_V \Ncov(A).\]
Let $W$ be a normed vector space of dimension $n$. Let $0 < \rho \leq 1$ be a parameter. If $\phi \colon V \to W$ is a $\inv{\rho}$-Lipschitz map, then 
\[\Ncov\bigl(\phi(A)\bigr) \ll_W \rho^{-n} \Ncov(A).\]

Since all our spaces are normed, we will need a notion of good bases : those whose vectors are well spaced. When $V$ is $\R^n$ or $\C^n$ endowed with an $l^2$-norm, recall that its norm induces an $l^2$-norm on each of its exterior powers. In this case the best bases are clearly orthonormal ones. Note that a basis $(a_1,\dotsc,a_n)$ is orthonormal if and only if $\forall k, \norm{a_k} \leq 1$ and $\norm{a_1\wedge \dotsb \wedge a_n} \geq 1$. If we loosen this condition, we get a notion of good bases. However, the norm on an exterior power of $V$ is properly defined only when $V$ is equipped with an $l^2$-norm and we will deal with other norms such as the operator norm on $\End(\R^n)$. Thus, we need an equivalent formulation.

\begin{lemm}\label{lm:rhoOrth}
Let $(a_1,\dotsc,a_n)$ be a basis of a normed vector space $V$ over $\R$ or $\C$, then the following conditions are equivalent in the sense that if the i-th condition holds for some $0 < \rho_i \leq 1$ then the j-th condition holds for some $\rho_j \gg_V \rho_i^{O(1)}$.
\begin{enumerate}
\item \label{it:alOrth} For all $k = 1,\dotsc,n$, $\norm{a_k} \leq \inv{\rho_1}$ and $d(a_k, \Span(a_1,\dotsc,a_{k-1})) \geq \rho_1$.
\item \label{it:alOrthx} For all $k = 1,\dotsc,n$, $\norm{a_k} \leq \inv{\rho_2}$ and all $x \in V$, its coordinates $(x_k)_k$ in the basis $(a_k)_k$ satisfy, $\forall k,\; \abs{x_k} \leq \rho_2^{-1}\norm{x}$.
\end{enumerate}
And moreover, if the norm on $V$ is an $l^2$-norm, then they are also equivalent to the following conditions.
\begin{enumerate}
\setcounter{enumi}{2}
\item \label{it:alOrthw} For all $k = 1,\dotsc,n$, $\norm{a_k} \leq \inv{\rho_3}$ and $\norm{a_1\wedge \dotsb \wedge a_n} \geq \rho_3$.
\item \label{it:alOrthl} Any endomorphism that maps an orthonormal basis to $(a_1,\dotsc,a_n)$ is $\rho_4^{-1}$-bi-Lipschitz.
\end{enumerate}
\end{lemm}

In condition \ref{it:alOrth}, we adhere to the convention that $\Span(\varnothing)$ means the zero subspace. This lemma is already known in \cite[Lemma 7.5]{EskinMozesOh} and \cite[Lemma 2.16]{Saxce}. We give an alternative proof.

\begin{proof}
Every norm on a finite dimensional linear space is equivalent to an $l^2$-norm. Hence it suffices to prove the equivalences in the case where $V= \R^n$ or $\C^n$ endowed with the standard norm. First, \ref{it:alOrth} implies \ref{it:alOrthw} since we have
\[\norm{a_1\wedge \dotsb \wedge a_n} = \prod_{k=1}^n d(a_k, \Span(a_1,\dotsc,a_{k-1})).\]

To see that \ref{it:alOrthw} implies \ref{it:alOrthx}, let $x \in E$, then $x = x_1a_1 + \dotsb + x_na_n$ with $(x_i)_i$ the coordinates of $x$ in $(a_i)_i$. On the one hand, 
\[\norm{x \wedge a_2 \wedge \dotsb \wedge a_n} \leq \norm{x}\norm{a_2}\dotsm \norm{a_n} \leq \rho_3^{-(n-1)}\norm{x}.\]
On the other hand, $x \wedge a_2\wedge \dotsb \wedge a_n = x_1 a_1\wedge \dotsb \wedge a_n$ so
\[\norm{x \wedge a_2 \wedge \dotsb \wedge a_n} = \abs{x_1}\, \norm{a_1\wedge \dotsb \wedge a_n} \geq \rho_3 \abs{x_1}.\]
Hence $\abs{x_1} \leq \rho_3^{-n} \norm{x}$ and the proof is similar for the other coordinates.

Equivalence between \ref{it:alOrthx} and \ref{it:alOrthl} is clear. 

Finally, \ref{it:alOrthl} implies \ref{it:alOrth} because the inequality in \ref{it:alOrth} holds for an orthonormal basis with $\rho_1 = 1$ and a $\rho_4^{-1}$-bi-Lipschitz map will only introduce a factor $\rho_4^{-1}$ or $\rho_4$ to these inequalities.
\end{proof}

\begin{rmq}
From the proof we see that the implied constant in the notation $\gg_{V}$ in the lemma can be $1$ if $V$ is endowed with an $l^2$-norm. Also, if $V_0$ is a fixed normed vector space, then this implied constant is uniform for all subspaces $V$ of $V_0$. 
\end{rmq}

Lemma~\ref{lm:rhoOrth} suggests the following definition.
\begin{déf}
Let $0 < \rho \leq 1$ be a parameter. We say a basis $(a_1,\dotsc,a_n)$ of a normed vector space $V$ is \emph{$\rho$-almost orthonormal} if it satisfies the condition \ref{it:alOrth} in Lemma~\ref{lm:rhoOrth} with $\rho_1 = \rho$.
\end{déf}

\begin{déf}
Let $0 < \rho \leq 1$ be a parameter. Let $V$ be a normed vector space. We say that a subset $A \subset V$ is \emph{$\rho$-away from linear subspaces} if for every proper linear subspace $W \subset V$, there is $a \in A$ such that $d(a,W) \geq \rho$.

Let $E$ be a normed algebra. We say that a subset $A \subset E$ is \emph{$\rho$-away from subalgebras} if for every proper subalgebra $W \subset E$, there is $a \in A$ such that $d(a,W) \geq \rho$.

In a similar way, we define the notion of being \emph{$\rho$-away from Lie subalgebras}.
\end{déf}

We have the following observation.
\begin{lemm}\label{lm:awayBasis}
Let $0 < \rho \leq \Inv{2}$ be a parameter. In a normed vector space $V$ of finite dimension, if a subset $A \subset \Ball(0,\rho^{-1})$ is $\rho$-away from linear subspaces, then $A$ contains a $\rho$-almost orthonormal basis. Conversely, if $A$ contains a $\rho$-almost orthonormal basis, then $A$ is $\rho^{O_V(1)}$-away from subspaces.
\end{lemm}

\begin{proof}
Assume that $A \subset \Ball(0,\rho^{-1})$ is $\rho$-away from linear subspaces. We can construct a $\rho$-basis from the set $A$ by induction. For $k = 1, \dotsc, \dim(V)$, suppose that $a_1,\dotsc,a_{k-1}$ are constructed, then  $\Span(a_1,\dotsc,a_{k-1})$ is a proper subspace of $V$. Hence there is $a_k \in A$ such that $d(a_k, \Span(a_1,\dotsc,a_{k-1})) \geq \rho$.

Conversely, assume that $A$ contains a $\rho$-almost orthonormal basis $(a_i)$. For any proper linear subspace $W \subset V$, there is $x \in V$ such that $\norm{x} = d(x,W) = 1$. By Lemma~\ref{lm:rhoOrth}, we can write $x = \sum_i x_i a_i$, with $\abs{x_i} \leq \rho^{-O_V(1)}$, for all $i$. Consequently,
\[d(x,W) \leq \sum_i \abs{x_i}d(a_i,W) \leq \rho^{-O_V(1)}\sum_i d(a_i,W).\]
Hence there is $i$ such that $d(a_i,W) \geq \rho^{O_V(1)}$.
\end{proof}

\begin{déf}
Let $0 < \rho \leq 1$ be a parameter. Let $A$ be a subset of $\End(\R^n)$. We say that $A$ \emph{acts $\rho$-irreducibly on $\R^n$} if for every proper nonzero linear subspace $W \subset \R^n$, there is $a \in A$ and $w \in W \cap \Ball(0,1)$ such that $d(aw,W) \geq \rho$.

We say that a subalgebra $E \subset \End(\R^n)$ \emph{acts $\rho$-irreducibly on $\R^n$} if the set $E\cap \Ball(0,1)$ acts $\rho$-irreducibly on $\R^n$.
\end{déf}

\subsection{Additive combinatorial tools}

Let us denote by $\OB(N)$ an unspecified finite set of cardinality $O(N)$. Let $\delta > 0$ be the scale. Each of the following combinatorial facts has a (better-known) counterpart in the discrete setting. We will need them in the discretized setting. Throughout this subsection, $A$ and $B$ denote bounded subsets of a normed vector space (or algebra when multiplication is involved). All implied constants in Landau and Vinogradov notations depend on the implicit ambient space.

\begin{lemm}[Ruzsa's covering lemma]\label{lm:RuzsaCov}
For all $0 < \rho < 1$, if $\Ncov(A + B) \leq \inv{\rho} \Ncov(A)$, then
\[B \subset A - A + \OB(\inv{\rho}) + \Ball(0,\delta).\]
\end{lemm}

\begin{proof}
Let $B_0$ be a maximal subset of $B$ such that the translates $\bigl(b + A + \Ball(0,\frac{\delta}{2})\bigr)_{b\in B_0}$ are disjoint. On the one hand, for every $b \in B$, the translate $b + A + \Ball(0,\frac{\delta}{2})$ is not disjoint from $b' + A + \Ball(0,\frac{\delta}{2})$ for some $b' \in B_0$ which means $b \in A - A + B_0 + \Ball(0,\delta)$. On the other hand, by the disjointness,
\[\Ncov(A + B) \gg \Ncov[\frac{\delta}{2}](A + B_0) \geq \abs{B_0} \Ncov[\frac{\delta}{2}](A).\]
Hence $\abs{B_0} \ll \inv{\rho}$.
\end{proof}

\begin{lemm}[Plünnecke-Ruzsa inequality]\label{lm:RuzsaSum}
For any parameter $0 < \rho \leq \Inv{2}$, if $\Ncov(A + A) \leq \inv{\rho} \Ncov(A)$ then for all integers $k \geq 0$ and $l \geq 0$,
\[\Ncov(kA - lA) \ll \rho^{-O(k+l)} \Ncov(A).\]
\end{lemm}

\begin{proof}
Consider the case where $A \subset \R^n$ with $l^\infty$-norm on $\R^n$.  We will approximate $A$ by points on the lattice $\delta \cdot \Z^n$. Let
\[A' = \ens{a \in \delta \cdot \Z^n \mid A \cap \bigl(a + \Ball(0,\delta)\bigr) \neq \varnothing}.\]
The sets $A$ and $A'$ are close in Hausdorff distance : $A \subset A' + \Ball(0,\delta)$ and $A' \subset A + \Ball(0,\delta)$. As a consequence, 
$A'+ A' \subset A + A + \Ball(0,2\delta)$, and hence 
\[\abs{A' + A'} \ll \Ncov(A+A) \leq \inv{\rho} \Ncov(A) \ll \inv{\rho}\abs{A'}.\]
Moreover, for any $k \geq 0$ and any $l \geq 0$, $kA - lA \subset kA' - lA' + \Ball(0,(k+l)\delta)$, and hence 
\[\Ncov(kA-lA) \ll_{k,l} \abs{kA'-lA'}.\]
We conclude the proof by applying the classical Plünnecke-Ruzsa inequality (see \cite[Proposition 2.26]{TaoVu} or \cite{Petridis}) to the set $A'$ in the discrete additive group $(\delta\cdot \Z^n, +)$
\end{proof}

The following lemma is the sum-product analogue of the "small tripling implies small n-pling" property in groups. This discretized version can be proved by mimicking the proof of its discrete counterpart in \cite[Lemma 5.5]{BreuillardIHP}.
\begin{lemm}\label{lm:AplusAA}
For any parameter $0 < \rho \leq \Inv{2}$, if $A \subset \Ball(0,\inv{\rho})$ and $\Ncov(A + A) + \Ncov(A + A \cdot A) \leq \inv{\rho}\Ncov(A)$, then for any positive integer $s$,
\[\Ncov(\sg{A}_s) \leq \rho^{-O_s(1)} \Ncov(A).\]
\end{lemm}

\begin{déf}
Let $n$ and $m$ be positive integers. For a bounded subset $A \subset \R^n$ and a Lipschitz map $\phi \colon \R^n \to \R^m$. Define the \emph{$\phi$-energy of $A$ at scale $\delta$} to be
\[\En_\delta(\phi,A) = \Ncov\bigl(\ens{(a,a') \in A \times A \mid \norm{\phi(a)-\phi(a')} \leq \delta}\bigr).\]
\end{déf}
Here we view $A \times A$ as a subset of $\R^n \times \R^n \simeq \R^{2n}$. Thus, the distance on $A \times A$ is induced by the Euclidean distance on $\R^{2n}$. 

\begin{lemm}\label{lm:phiEnergy}
Let $n$ and $m$ be positive integers. Let $A \subset \R^n$ be a bounded subset and $\phi \colon \R^n \to \R^m$ a $\inv{\rho}$-Lipschitz map.
\begin{enumerate}
\item We have
\[ \Ncov\bigl(\phi(A)\bigr) \gg \frac{\Ncov(A)^2}{\En_\delta(\phi,A)}.\]
\item Let $\tilde A$ be a maximal $\delta$-separated subset of $A$. Then
\[\En_\delta(\phi,A) \ll \#\ens{(a,a') \in \tilde A \times \tilde A \mid \norm{\phi(a) - \phi(a')} \leq (1 + 2\inv{\rho})\delta}.\]
\end{enumerate}
\end{lemm}

\begin{proof}
Let $\tilde A$ be a maximal $\delta$-separated subset of $A$. 
\begin{enumerate}
\item Let $Y$ be a finite subset of $\R^m$ such that $\phi(A)$ is covered by the balls of radius $\frac{\delta}{2}$ centered at points in $Y$. Then
\[ \abs{\tilde A} \leq \sum_{y \in Y} \bigl|\tilde A \cap \inv{\phi}\bigl(\Ball(y,\frac{\delta}{2})\bigr)\bigr|,\]
and 
\[\sum_{y \in Y} \bigl|\tilde A \cap \inv{\phi}\bigl(\Ball(y,\frac{\delta}{2})\bigr)\bigr|^2 \ll \En_\delta(\phi,A).\]
It follows from the Cauchy-Schwarz inequality that 
\[\abs{Y} \gg \frac{\abs{\tilde A}^2}{\En_\delta(\phi,A)}.\]
\item For each $a \in A$, choose $\tilde a \in \tilde A$ such that $\norm{a - \tilde a} \leq \delta$. Let $\Omega$ be a maximal $8\delta$-separated subset of $\ens{(a,a') \in A \times A \mid \norm{\phi(a)-\phi(a')} \leq \delta}$. Then the map $(a,a') \mapsto (\tilde a,\tilde a')$ is injective from $\Omega$ to the set on the right-hand side of the desired inequality.\qedhere
\end{enumerate}

\end{proof}

\subsection{\L{}ojasiewicz inequality}
The \L{}ojasiewicz inequality~\cite[Théorème 2, page 62]{Lojasiewicz} is a powerful tool which allows us to extract quantitative estimates from algebraic facts. Let us recall it here. Consider 

\begin{thm}[\L{}ojasiewicz inequality]
\label{thm:Lojasiewicz}
Let $M$ be a real analytic manifold endowed with a Riemannian distance $d$. Let $f \colon M \to \R$ be a real analytic map. If $K$ is a compact subset of $M$, then there is $C > 0$ depending on $K$ and $f$ such that for all $x \in K$,
\[\abs{f(x)} \geq \Inv{C}\min\{d(x,Z)^C,1\}\]
where $Z = \ens{x \in M \mid f(x) = 0}$.
\end{thm}

Above we adhere to the convention that $\min\{d(x,Z)^C,1\} = 1$ if $Z$ is the empty set. This theorem is stated in~\cite{Lojasiewicz} for $M$ which are open sets of Euclidean spaces. To see that is also valid for real analytic manifolds endowed with a Riemannian distance, it suffices to read the map $f$ in coordinate charts and observe that a coordinate chart of a Riemannian manifold is necessarily bi-Lipschitz to its image (endowed with the distance induced from the Euclidean space).

\section{Escaping from subvarieties}\label{sc:escape}
In this section we show that if a subset $A$ of a simple algebra is not trapped in any subalgebra then we can escape from any subvariety within a bounded number of steps using addition and multiplication. The number of necessary steps depends only on the subvariety and the ambient algebra. This is achieved in two steps. First, using only multiplication we can escape from linear subspaces (Proposition~\ref{pr:escapeV}). Then, once the set is away from linear subspaces, we can escape from subvarieties using only addition (Lemma~\ref{lm:escapeSV}). Note that everything is quantitative. By escaping from a subvariety we mean getting outside a neighborhood of that subvariety.

\subsection{Escaping from linear subspaces} Let $A$ be a subset of a normed algebra $E$. Obviously, if $A$ is away from linear subspaces, then it is away from subalgebras. We will see in this subsection that the converse is true if we are allowed to replace $A$ with its product set $A^s$.

The following is the subalgebra version of \cite[Lemma 2.5]{Saxce}. The proof is essentially the same.
\begin{lemm}\label{lm:FiniteA}
Let $0 < \rho \leq \Inv{2}$ be a parameter. Let $A$ be a subset of a normed algebra $E$ of finite dimension. If $A \subset \Ball(0,\inv{\rho})$ and $A$ is $\rho$-away from subalgebras, then $A$ contains a subset of cardinality at most $\dim(E)$ which is $\rho^{O_E(1)}$-away from subalgebras.
\end{lemm}

\begin{proof}
Let $C \geq 1$ be a large constant. Suppose that $a_1,\dotsc ,a_{k-1}$ are constructed. If $\ens{a_1,\dotsc, a_{k-1}}$ is $\rho^C$-away from subalgebras, then we are done. Otherwise there is a proper subalgebra $W$ such that for all $i = 1, \dotsc, k - 1$, $d(a_i,W) < \rho^C$. Then choose from $A$ an element $a_k$ such that $d(a_k,W) \geq \rho$.

We prove by induction that at each step $(a_1,\dotsc ,a_k)$ is a $\rho^{O_E(1)}$-almost orthonormal basis of its linear span $V_k=\Span(a_1,\dotsc ,a_k)$. This is obvious for $k=1$. For $k \geq 2$, suppose that $(a_1,\dotsc,a_{k-1})$ is a $\rho^{O_E(1)}$-almost orthonormal basis of $V_{k-1}$. We can write
$a_k = x_1a_1 + \dotsb + x_{k-1} a_{k-1}+ a'_k$ with $x_i \in \R$ and $\norm{a'_k} = d(a_k,V_{k-1})$. But $x_1a_1 + \dotsb + x_{k-1} a_{k-1}$ is the decomposition of the vector $a_k - a'_k$ in the $\rho^{O_E(1)}$-almost orthonormal basis $(a_1,\dotsc,a_{k-1})$. Hence for every $i = 1, \dotsc k-1$, $\abs{x_i} \leq 2\rho^{-O_E(1)}\norm{a_k} \leq \rho^{-O_E(1)}$. And then
\begin{align*} 
\rho \leq d(a_k,W) &\leq \sum_{i = 1}^{k-1}\abs{x_i} d(a_i,W)  + \norm{a'_k}\\
&\leq \rho^{-O_E(1)}\max_{1 \leq i \leq k-1}d(a_i,W) + \norm{a'_k}\\
& \leq \rho^{C-O_E(1)} + \norm{a'_k}
\end{align*}
When $C$ is large enough, this implies $d(a_k,V_{k-1}) = \norm{a'_k} \geq \rho^2$. It follows that $(a_1,\dotsc ,a_k)$ is a $\rho^{O_E(1)}$-almost orthonormal basis of $V_k$.

We conclude that the construction must stop after at most $\dim(E)$ steps.
\end{proof}

Then we have the analogue of \cite[Proposition 2.7]{Saxce}.
\begin{lemm}\label{lm:awWwaW}
Let $0 < \rho \leq \frac{1}{2}$ be a parameter. Let $A$ be a subset of a normed simple algebra $E$ of finite dimension. If $A \subset \Ball(0,\inv{\rho})$ and $A$ is $\rho$-away from subalgebras, then for every proper nonzero linear subspace $W$ of $E$, there is $w \in W \cap \Ball(0,1)$ and $a \in A$ such that
\[d(aw,W) \geq \rho^{O_E(1)} \quad \text{or} \quad d(wa,W) \geq \rho^{O_E(1)}.\]
\end{lemm}

\begin{proof}
In view of Lemma~\ref{lm:FiniteA}, we can assume that $A$ has exactly $n = \dim(E)$ elements $a_1,\dotsc, a_n$. We can further assume that $A \subset \Ball(0,1)$ for we can replace $A$ with a contraction $\rho \cdot A$. We will treat the case where the norm on $E$ is Euclidean. The general case follows easily since every norm on $E$ is equivalent to an Euclidean one. Suppose the lemma were false. Then there would be a linear subspace $W_0$ of dimension $0 < k < n$ such that for all $w \in W_0 \cap \Ball(0,1)$ and all $i = 1,\dotsc, n$, $d(a_iw,W_0) < \rho^C$ and $d(wa_i,W_0) < \rho^C$ for some large $C$. The actual value of this constant will be determined by the \L{}ojasiewicz inequality used below.

Consider the map $f \colon \Gr(k,E) \times E^n \to \R$ defined by
\begin{equation*}
f(W;x_1,\dotsc,x_n) = \displaystyle\sum_{i=1}^n\int_{\raisebox{-0.4ex}{$\scriptstyle W\cap \Ball(0,1)$}} \hspace{-2em} d(x_iw,W)^2 + d(wx_i,W)^2 \,\dd w
\end{equation*}
where the integration $\dd w$ is with respect to the $k$-dimensional Lebesgue measure on $W \cap \Ball(0,1)$. This map is well-defined and real analytic. This can be seen by observing that the tautological bundle $\pi \colon T \to \Gr(k,E)$ of the Grassmannian has around every point a real analytic local trivialization $\phi \colon \UC \times \R^k \to \inv{\pi}(\UC)$ such that $\forall W \in \UC$, $\phi(W,\bullet) \colon \R^k \to \inv{\pi}(\{W\}) \simeq W$ is an isometry of Euclidean spaces.

From the choice of $W_0$ it follows that $f(W_0;a_1,\dotsc,a_n) \ll \rho^C$. Hence by the \L{}ojasiewicz inequality (Theorem~\ref{thm:Lojasiewicz}) applied to the compact set $\Gr(k,E) \times \Ball(0,1)^n$, there is $W_1 \in \Gr(k,E)$ and $b_1,\dotsc,b_n \in E$ such that $f(W_1;b_1,\dotsc,b_n) = 0$ and $\forall i = 1, \dotsc, n$, $\norm{a_i - b_i} < \rho$ if the constant $C$ is chosen large enough. The map $f$ vanishing at $(W_1;b_1,\dotsc,b_n)$ is equivalent to every $b_i$ being in the subalgebra
\[E_{W_1} = \ens{ x\in E \mid xW_1 \subset W_1 \text{ and } W_1x \subset W_1}.\]

Now our set $A$ is not $\rho$-away from the subalgebra $E_{W_1}$. Hence $E_{W_1}$ must be the whole algebra $E$, which in turn implies that $W_1$ is a two-sided ideal in $E$. This contradicts the assumption that $E$ is simple.
\end{proof}

\begin{prop}\label{pr:escapeV}
Let $0 < \rho \leq \frac{1}{2}$ be a parameter. Let $A$ be a subset of a normed simple algebra $E$ of dimension $n$. Assume that $A \subset \Ball(0,\inv{\rho})$ and $A$ is $\rho$-away from subalgebras. Write $A_1= \ens{1_E} \cup A$. Then for any $\eta \in E$ with $\rho \leq \norm{\eta} \leq \inv{\rho}$, the set $A_1^{n}\eta A_1^{n}$ contains a $\rho^{O_E(1)}$-almost orthonormal basis of $E$. And, equivalently, $A_1^{n}\eta A_1^{n}$ is $\rho^{O_E(1)}$-away from linear subspaces in $E$. 
\end{prop}

\begin{proof}
We construct the basis inductively. First, let $\eta_1 = \eta$. Then for $k = 1,\dotsc,n-1$, suppose that after $k$ steps, we have constructed $\eta_1,\dotsc,\eta_k \in A^k\eta A^k$ such that $(\eta_1,\dotsc,\eta_k)$ is a $\rho^{O_E(1)}$-almost orthonormal basis of $W_k = \Span(\eta_1,\dotsc,\eta_k)$. 
By Lemma~\ref{lm:awWwaW}, there is $w \in W_k \cap \Ball(0,1)$ and $a \in A$ such that either $d(aw,W) \geq \rho^{O_E(1)}$ or $d(wa,W) \geq \rho^{O_E(1)}$. Let us deal with the former case, the latter case being similar. We can write $w$ in the $\rho^{O_E(1)}$-almost orthonormal basis $(\eta_1,\dotsc,\eta_k)$, 
\[w = w_1\eta_1 + \dotsb + w_k\eta_k.\] 
By Lemma~\ref{lm:rhoOrth}, the coefficients $w_i$ satisfy $\forall i$, $\abs{w_i} \leq \rho^{-O_E(1)}$. Hence
\[d(aw,W_k) \leq \sum_{i=1}^k \abs{w_i}d(a\eta_i,W_k) \leq \rho^{-O_E(1)}\max_{i=1,\dotsc,k} d(a\eta_i,W_k).\] 
Hence it is possible to pick $i_* \in \ensA{k}$ so that, writing $\eta_{k+1} = a\eta_{i_*}$, we have $d(\eta_{k+1},W_k) \geq \rho^{O_E(1)}$. Then $(\eta_1,\dotsc,\eta_{k+1})$ is a $\rho^{O_E(1)}$-almost orthonormal basis of its linear span $W_{k+1}$. This finishes the proof of the inductive step.
\end{proof}

\subsection{Escaping from subvarieties, \texorpdfstring{$\R^n$}{R^n} case}

In this subsection we show how to escape from subvarieties using addition starting from a set which is away from linear subspaces. Clearly it suffices to deal with the case of a hypersurface. Note that if the hypersurface is defined by a polynomial $P$ then a point $a \in \R^n$ is far away from it if $\abs{P(a)}$ is large. Here is the statement.

\begin{lemm}\label{lm:escapeSV}
Let $V$ be a normed vector space and $P \colon V \to \R$ a nonzero polynomial map. Then there is an positive integer $s$ depending only on the degree of $P$ and the dimension of $V$ such that the following holds for all parameters $0 < \rho \leq \Inv{2}$. Let $A$ be a subset of $V$ with $A \subset \Ball(0,\rho^{-1})$. If $A$ is $\rho$-away from linear subspaces then there is $a \in s(A\cup\ens{0})$ such that
\[\abs{P(a)} \gg_{V,P} \rho^{O_{V,P}(1)}.\]
\end{lemm}

It is a straightforward consequence of Lemma~\ref{lm:awayBasis} and the following lemma.

\begin{lemm}
Let $d \geq 1$ and $0 < \rho < 1$. Let $(a_1,\dotsc,a_n)$ be a $\rho$-almost orthonormal basis of $\R^n$. Define
\[\Gamma_d = \bigl\{x_1a_1 + \dotsb + x_na_n \mid \forall i,\; x_i \in \ens{0,\dotsc,d}\bigr\}.\]
For any polynomial function $P \colon \R^n \to \R$ of degree at most $d$, we have
\[\sup_{x \in \Ball(0,1)}\abs{P(x)} \ll_{n,d} \rho^{-O_{n,d}(1)}\max_{\gamma \in \Gamma_d}\abs{P(\gamma)}.\]
\end{lemm}

\begin{proof}
We write $\R_d[X_i]_{1\leq i \leq n}$ for the space of polynomials on $\R^n$ of degree at most $d$. First consider the case where $(a_1,\dotsc,a_n)$ is the standard basis. Let
\[\Lambda_d = \bigl\{(x_1,\dotsc,x_n) \in \R^n \mid \forall i,\; x_i \in \ens{0,\dotsc,d}\bigr\}.\]
By a simple induction on the number of variables $n$, we see that if $P \in \R_d[X_i]_{1\leq i \leq n}$ and vanishes on $\Lambda_d$ then it must be the zero polynomial. Therefore the linear map
\begin{equation*}
\begin{matrix}
  \R_d[X_i]_{1\leq i \leq n} &\to& \R^{\Lambda_d}\\
  P &\mapsto & (P(\lambda))_{\lambda \in \Lambda_d},
\end{matrix}
\end{equation*}
is injective. Hence the coefficients of $P$ are controlled by $\max_{\lambda \in \Lambda_d}\abs{P(\lambda)}$ and consequently for all $r \geq 1$,
\begin{equation}\label{eq:spPgamma}
\sup_{x \in \Ball(0,r)}\abs{P(x)} \ll_{n,d} r^d \max_{\lambda\in \Lambda_d}\abs{P(\lambda)}.
\end{equation}

For the general case consider $\psi \colon \R^n \to \R^n$ the unique linear map which sends the standard basis to $(a_1,\dotsc,a_n)$. Thus $\psi(\Lambda_d) = \Gamma_d$ and, by Lemma~\ref{lm:rhoOrth}, $\psi(\Ball(0,\rho^{-O(1)})) \supseteq \Ball(0,1)$. We obtain the desired estimate by applying inequality~\eqref{eq:spPgamma} to the polynomial $P\circ\psi$ and $r = \rho^{-O(1)}$.
\end{proof}

\section{Trace set estimates}\label{sc:trace}

To each finite-dimensional real algebra $E$ we can associate a trace function $\tr_E$ and a bilinear form $\tau_E$. Let $x$ be an element of $E$, we define $\tr_E(x)$ to be the trace of the left multiplication by $x$ as an endomorphism of $E$.
For example, if $E = \Mat_n(\R)$, then $\tr_E$ is $n$ times the usual trace for matrices.

Let $x,y \in E$. We define $\tau_E(x,y) = \tr_E(xy)$. Thus $\tau_E \colon E \times E \to \R$ is a symmetric bilinear form. Observe that its kernel $\ker \tau_E$ is a two-sided ideal of $\R$. It follows that $\tau_E$ is non-degenerate if $E$ is semisimple, i.e. direct sum of simple algebras. Note that the converse is also true, but we won't need this fact here.

Let $A$ be a bounded subset of a semisimple algebra $E$. In this section, we are interested in the size of the trace set of $A^2$ : 
\[\tr_E(A^2) = \ens{\tr_E(ab) \mid a , b\in A}.\]
The aim is to establish a lower bound under appropriate conditions. Here is the result.

\begin{lemm}\label{lm:traceSet}
Given $\epsilon > 0$, the following is true for sufficiently small $\delta > 0$. Let $E$ be a normed semisimple algebra. Let $A$ be a subset of $E$ such that
\begin{itemize}
\item $A \subset \Ball(0,\delta^{-\epsilon})$,
\item $A$ is $\delta^\epsilon$-away from linear subspaces in $E$.
\end{itemize}
Then $\Ncov(\tr_E(A^2)) \gg_E \delta^{O(\epsilon)}\Ncov(A)^{\Inv{n}}$, where $n = \dim(E)$.
\end{lemm}

We begin with a lemma.

\begin{lemm}\label{lm:biLin}
Let $0 < \rho \leq \Inv{2}$ be a parameter. Let $V$ be a normed vector space of dimension $n$ and $\tau \colon V \times V \to \R$ a non-degenerate bilinear form. If $(a_i)$ is a $\rho$-almost orthonormal basis of $V$, then for all $x \in V$,
\[\norm{x} \ll_{V,\tau} \rho^{-O(1)} \max_i \abs{\tau(a_i,x)}.\]
\end{lemm}

The $V$ and $\tau$ in the subscript indicate that the implied constant depends not only on $V$ but also on the bilinear form.
 
\begin{proof}
It suffices to deal with the special case where $V = \R^n$ endowed with the standard Euclidean norm. Let $g$ be the endomorphism which maps the standard basis to the $\rho$-almost orthonormal basis $(a_i)$. By Lemma~\ref{lm:rhoOrth}, $\norm{\inv{g}} \leq \rho^{-O(1)}$. Let $s$ be the matrix of $\tau$ in the standard basis. Then a straightforward computation shows that for any $x \in V$, $\prescript{t}{}{\!g}sx$ is the column vector $(\tau(a_i,x))_i$. Hence
\[\norm{x} \leq \norm{\inv{s}} \norm{\prescript{t}{}{\!g}^{-1}} \norm{\prescript{t}{}{\!g}sx} \ll_{n,\tau} \rho^{-O(1)} \max_i \abs{\tau(a_i,x)}.\qedhere\] 
\end{proof}

\begin{proof}[Proof of Lemma \ref{lm:traceSet}]
By Lemma~\ref{lm:awayBasis}, $A$ contains a $\delta^\epsilon$-almost orthonormal basis $(a_i)_{1 \leq i \leq n}$ of $E$. Thus, by Lemma~\ref{lm:biLin} applied to the non-degenerate bilinear form $\tau_E$, for all $x \in E$, $\norm{x} < \delta^{1-O(\epsilon)}$ whenever $\abs{\tr_E(a_ix)} < \delta$ for all $i = 1,\dotsc,n$. 

Consider the map $\Theta \colon A \to \R^n$ defined by $x \mapsto \bigl(\tr_E(a_ix) \bigr)_i$. On the one hand, this map is "almost injective" with the $\delta$-blurred vision : for all $x,y \in A$, if $\norm{\Theta(x)- \Theta(y)} < \delta$ then $\norm{x - y} < \delta^{1-O(\epsilon)}$. It follows that 
\[\Ncov(\Theta(A)) \gg_n \Ncov[\delta^{1-O(\epsilon)}](A) \gg_E \delta^{O(\epsilon)} \Ncov(A).\]
On the other hand, the set $\Theta(A)$ is contained in the $n$-fold Cartesian product $\tr(A^2) \times \dotsb \times \tr(A^2)$, hence 
\[\Ncov(\Theta(A)) \ll_n \Ncov(\tr(A^2))^n.\]
We obtain the desired estimate by combining the two inequalities above.
\end{proof}

\section{Effective Wedderburn theorem}\label{sc:Wedderburn}

The Wedderburn theorem (see \cite[Chapter XVII, \S 3]{Lang_algebra} or \cite{Serre50}) states that if $\R^n$ is an irreducible representation of an algebra $E$, then $E$ is isomorphic to either $\Mat_n(\R)$ or $\Mat_{\frac{n}{2}}(\C)$ or $\Mat_{\frac{n}{4}}(\HH)$. In this section we prove a quantified version of this algebraic fact. From now on, we endow these matrix algebras with operator norms\footnote{Each matrix is seen as an endomorphism of the Euclidean space $\R^n$, $\C^n\simeq \R^{2n}$ or $\HH^n \simeq \R^{4n}$.}. And throughout this paper, $\Mat_{\frac{n}{2}}(\C)_\R$ denotes a fixed embedding of $\Mat_{\frac{n}{2}}(\C)$ in $\Mat_n(\R)$ and $\Mat_{\frac{n}{4}}(\HH)_\R$ a fixed embedding of $\Mat_{\frac{n}{4}}(\HH)$ in $\Mat_n(\R)$.
\begin{prop}\label{pr:Wedder}
Let $0< \rho \leq \Inv{2}$ be a parameter. Let $E$ be a subalgebra of $\Mat_n(\R)$ acting $\rho$-irreducibly on $\R^n$. Then either $E = \Mat_n(\R)$ or $E$ is conjugate to $\Mat_{\frac{n}{2}}(\C)_\R$ or to $\Mat_{\frac{n}{4}}(\HH)_\R$ by a change of basis matrix $g \in \GL(\R^n)$ satisfying $\norm{g} + \norm{\inv{g}} \leq \rho^{-O(1)}$.

In particular, $E$ is isomorphic to $\Mat_n(\R)$ or $\Mat_{\frac{n}{2}}(\C)$ or $\Mat_{\frac{n}{4}}(\HH)$ by an isomorphism which is $\rho^{-O(1)}$-bi-Lipschitz.
\end{prop}

This result will be useful in the proof of Theorem~\ref{thm:ActionRn}.

\subsection{Effective diagonalization} 
The following lemma is implicit in \cite[proof of Proposition 7.4]{EskinMozesOh}. We include its proof for the sake of completeness.

\begin{lemm}\label{lm:effDiag}
Let $0 < \rho \leq \Inv{2}$ be a parameter. Let $a \in \Mat_n(\C)$ be a diagonalizable matrix. Assume that $\norm{a} \leq \rho^{-1}$ and its spectrum is a $\rho$-separated set (it may contain multiple eigenvalues). Then $a$ is diagonal in a $\rho^{O(1)}$-almost orthonormal basis of $\C^n$.
\end{lemm}

\begin{proof}
Let $\lambda_1,\dotsc ,\lambda_n$ be the eigenvalues of $a$ (they appear with the corresponding multiplicity). Let $v_1, \dotsc, v_n$ be the corresponding eigenvectors. Each $v_i$ can be chosen to be of unit length and we can further assume that $(v_i,v_j) = 0$ whenever $i \neq j$ but $\lambda_i = \lambda_j$. For $k = 1, \dotsc, n$, write $V_k = \Span(v_1,\dotsc,v_k)$. We will prove by induction on $k$ that $(v_1,\dots,v_k)$ is a $\rho^{O_k(1)}$-almost orthonormal basis of $V_k$. This is clear for $k = 1$. 

Let $k \in \ensA{n-1}$. Suppose that $(v_1,\dots,v_k)$ is a $\rho^{O_k(1)}$-almost orthonormal basis of $V_k$. Let us show that $d(v_{k+1},V_k) \geq \rho^{O_k(1)}$ and thus $(v_1,\dots,v_k, v_{k+1})$ is a $\rho^{O_k(1)}$-almost orthonormal basis of $V_{k+1}$. Without loss of generality we can assume that among the the eigenvalues $\lambda_1, \dotsc, \lambda_k$, only $\lambda_{l+1},\dotsc, \lambda_k$ are equal to $\lambda_{k+1}$ for some $l \leq k$. Considering the orthogonal projection of $v_{k+1}$ onto $V_k$, we can decompose the vector $v_{k+1}$ into
\begin{equation}\label{eq:vjp1Decom}
v_{k+1} = \sum_{i=1}^k x_iv_i + v'_{k+1},
\end{equation}
where $x_1,\dotsc,x_k \in \C$ and $v'_{k+1} \in V_k^{\perp}$. In particular, $d(v_{k+1},V_k) = \norm{v'_{k+1}}$. We can then express $av_{k+1}$ in two different ways :
\[a v_{k+1} = \lambda_{k+1}v_{k+1} = \sum_{i=1}^k \lambda_{k+1}x_iv_i + \lambda_{k+1}v'_{k+1},\]
and
\[a v_{k+1} = \sum_{i=1}^k \lambda_ix_iv_i + av'_{k+1}.\]
It follows that
\[\sum_{i=1}^l (\lambda_{k+1} - \lambda_i)x_iv_l = \lambda_{k+1}v'_{k+1} - av'_{k+1}.\]

Denote by $w_{k+1}$ the vector on the right-hand side. Its norm can be bounded : $\norm{w_{k+1}} \leq 2\norm{a}\norm{v'_{k+1}} \leq \rho^{-2}\norm{v'_{k+1}}$. The left-hand side gives the coordinates of $w_{k+1}$ in the basis $(v_1,\dots,v_k)$ which is $\rho^{O_k(1)}$-almost orthonormal by the induction hypothesis. Hence, by Lemma~\ref{lm:rhoOrth}, 
\[\forall i = 1,\dotsc,l,\quad \abs{\lambda_{k+1} -\lambda_{i}}\abs{x_i} \leq \rho^{-O_k(1)}\norm{w_{k+1}} \leq \rho^{-O_k(1)}\norm{v'_{k+1}}.\]
Hence for all $i = 1,\dotsc, l$, $\abs{x_i} \leq \rho^{-O_k(1)}\norm{v'_{k+1}}$ thanks to the assumption $\abs{\lambda_{k+1}-\lambda_i} \geq \rho$. 

In order to bound $\abs{x_j}$ for $j \geq l+1$, we take the scalar product with $v_j$ on both sides of \eqref{eq:vjp1Decom}. We obtain
\[ 0 = \sum_{i=1}^l x_i(v_j,v_i) + x_j.\]
Hence for all $j = l+1,\dots k$, $\abs{x_j} \leq \rho^{-O_k(1)}\norm{v'_{k+1}}$.
Using \eqref{eq:vjp1Decom} again we obtain $1 = \norm{v_{k+1}} \ll \rho^{-O_k(1)}\norm{v'_{k+1}}$ and then $\norm{v'_{k+1}} \geq \rho^{O_k(1)}$. This finishes the proof of the inductive step.
\end{proof}

\subsection{Effective Wedderburn theorem in \texorpdfstring{$\Mat_n(\R)$}{M_n(R)}} 

For $x \in \Mat_n(\R)$ and $A \subset \Mat_n(\R)$ denote by $\CC(x)$ and $\CC(A)$ their respective centralizers, i.e.
\[\CC(x) = \ens{y \in \Mat_n(\R) \mid xy = yx},\]
\[\CC(A) = \ens{y \in \Mat_n(\R) \mid \forall x \in A, xy = yx}.\]

\begin{lemm}\label{lm:CentFix}
Let $0 < \rho \leq 1$ be a parameter. For all $g \in \SL_n(\R)$, if $\norm{g} > \rho^{-n}$, then the subalgebra $\CC(g)$ does \emph{not} act $\rho$-irreducibly on $\R^n$.
\end{lemm}

\begin{proof}
Let $g = kal$ be the Cartan decomposition of $g$, with $k,l \in \SO(n)$ and $a = \diag(a_1,\dotsc,a_n)$ where its singular values $a_1,\dotsc, a_n$ are arranged so that $a_1 \geq a_2 \geq \dots \geq a_n > 0$. Assume that $a_1 = \norm{g} > \rho^{-n}$. Since $a_1 \dotsm a_n = \det(g) = 1$, there is $p \in \ensA{n-1}$ such that 
\[\frac{a_{p+1}}{a_p} < \rho.\]
Put $W = k\Span(e_1,\dotsc,e_p)$. It is a nonzero proper linear subspace of $\R^n$. We claim that for all $x \in \CC(g)\cap \Ball(0,1)$ and all $w \in W \cap \Ball(0,1)$, $d(xw,W) < \rho$. 

Indeed, decomposing every vector $v \in \R^n$ as $v = v' + v''$ with $v' \in \inv{l}\Span(e_1,\dotsc,e_p)$ and $v'' \in \inv{l}\Span(e_{p+1},\dotsc e_n)$, we see that
\[d(gv,W) = \norm{av''} \leq a_{p+1} \norm{v}.\]
Moreover, for all $w \in W$, 
\[\norm{\inv{g}w} \leq a_p^{-1}\norm{w}.\]
 
Consequently, for all $x \in \CC(g) \cap \Ball(0,1)$ and all $w \in W \cap \Ball(0,1)$, we have $xw = gx\inv{g}w$, and thus
\[d(xw,W) \leq a_{p+1}\norm{x\inv{g}w} \leq a_{p+1} \norm{\inv{g}w} \leq \frac{a_{p+1}}{a_p} < \rho. \qedhere\]
\end{proof}

\begin{lemm}\label{lm:realConj}
Let $0 < \rho \leq \Inv{2}$ be a parameter. If two real matrices $x$ and $y \in \Mat_n(\R)$ are conjugate by a complex matrix $g \in \GL_n(\C)$ with $\norm{g} + \norm{\inv{g}} < \inv{\rho}$ then there is a real matrix $h \in \GL_n(\R)$ which also conjugates them and moreover $\norm{h} + \norm{\inv{h}} \leq \rho^{-O(1)}$. 
\end{lemm}

\begin{proof}
Assume that $gx\inv{g} = y$ with $x, y \in \Mat_n(\R)$ and $g \in \GL_n(\C)$. We write $g = g_{\Re} + ig_{\Im}$ with $g_\Re, g_\Im \in \Mat_n(\R)$ its real and imaginary part. From $gx = yg$ we see that $g_\Re x= yg_\Re$ and $g_\Im x= yg_\Im$. For $\lambda \in \C$, consider $h_\lambda = g_{\Re} + \lambda g_{\Im}$. For all $\lambda \in \C$, we have $h_\lambda x = y h_\lambda$. Hence whenever $\det(h_\lambda) \neq 0$, $h_\lambda$ conjugates $x$ and $y$. What remains to do is to find appropriate $\lambda \in \R$ such that $h_\lambda$ and $h_\lambda^{-1}$ have bounded norms.

Define $P(\lambda) = \det(h_\lambda)$. It is a polynomial with real coefficients and its degree is at most $n$. We know that 
\[\abs{P(i)} = \abs{\det(g)} = \abs{\det(\inv{g})}^{-1} \gg \norm{g^{-1}}^{O(1)} \gg \rho^{O(1)}.\]
It is easy to see that the coefficients of $P$ are controlled by $\max_{\lambda \in [0,1]} \abs{P(\lambda)}$. Hence
\[\abs{P(i)} \ll \max_{\lambda \in [0,1]} \abs{P(\lambda)}.\]
And so there is $\lambda_0 \in [0,1]$ such that $P(\lambda_0) \gg \rho^{O(1)}$. Take $h = h_{\lambda_0}$. we have
\[\norm{h} \leq \norm{g_\Re} + \abs{\lambda_0}\norm{g_\Im} \leq 2\norm{g} \ll \inv\rho.\]
and $\det(h) \gg \rho^{O(1)}$, which implies
\[\norm{\inv{h}} \leq \norm{h}^{n-1}\det(h)^{-1} \ll \rho^{-O(1)}.\]
Here the first inequality can be seen from the Cartan decomposition of $h$. 
\end{proof}

\begin{proof}[Proof of Proposition~\ref{pr:Wedder}]
Consider $K = \CC(E)$ the centralizer of $E$ in $\Mat_n(\R)$. From (the proof of) Wedderburn's theorem, $K$ is a division algebra over $\R$ and $E$ is equal to $\CC(K)$, the centralizer of $K$. By the Frobenius theorem, the real division algebra $K$ is isomorphic to $\R$, $\C$ or $\HH$. The action of $K$ on $\R^n$ makes $\R^n$ a $K$ linear space. Hence $n$ is even if $K \simeq \C$ and a multiple of $4$ if $K \simeq \HH$. 

If $K \simeq \R$ then $E = \Mat_n(\R)$, we are done. If $K \simeq \C$ or respectively if $K \simeq \HH$, then $E$ is isomorphic to $\Mat_{\frac{n}{2}}(\C)$ or respectively to $\Mat_{\frac{n}{4}}(\HH)$. The Skolem-Noether theorem (see \cite[\S 12.6]{Pierce}) tells us all embeddings of $\Mat_{\frac{n}{2}}(\C)$ or $\Mat_{\frac{n}{4}}(\HH)$ in $\Mat_n(\R)$ are conjugate (by a matrix in $\GL(\R^n)$). We are going to show that under our quantitative irreducibility assumption, the change of basis matrix can be nicely chosen. 

From the discussion above, we see that an embedding of $\Mat_{\frac{n}{2}}(\C)$ is uniquely determined by an embedding of $\C$ in $\Mat_n(\R)$. Moreover, if $K$ is conjugate to a fixed embedding of $\C$, then $E$ is conjugate to a fixed embedding of $\Mat_{\frac{n}{2}}(\C)$ by the same change of basis matrix. The same applies to the quaternion case. That is why we are going to study embeddings of $\C$ and of $\HH$. Let $\varphi_0 \colon \C \to \Mat_n(\R)$ denote the embedding of $\C$ such that $\CC(\varphi_0(\C)) = \Mat_{\frac{n}{2}}(\C)_\R$ and $\psi_0 \colon \HH \to \Mat_n(\R)$ the embedding of $\HH$ such that $\CC(\psi_0(\HH)) = \Mat_{\frac{n}{4}}(\HH)_\R$.

When $K$ is isomorphic to $\C$, denote by $\varphi \colon \C \to \Mat_n(\R)$ the embedding of $\C$ whose image is $K$. The homomorphism $\varphi$ is uniquely determined by $\varphi(i)$ and we have $E = \CC(\varphi(i))$. From Lemma~\ref{lm:CentFix} and the irreducibility assumptions on $E$, we have $\norm{\varphi(i)} \ll \rho^{-O(1)}$. The matrix $\varphi(i)$ satisfies $\varphi(i)^2 = - \Id_n$, which implies that $\varphi(i)$ is diagonalizable over $\C$ and its eigenvalues can only be $i$ or $-i$. Its trace is real. Therefore the multiplicities of the two eigenvalues must be equal. We conclude that $\varphi(i)$ is conjugate (over $\C$) to the diagonal matrix $\diag(i\Id_{\frac{n}{2}},-i\Id_{\frac{n}{2}})$. By Lemma~\ref{lm:effDiag}, $\varphi(i)$ is conjugate to this matrix by a change of basis matrix $g \in \GL_n(\C)$ satisfying $\norm{g} + \norm{\inv{g}} \ll \rho^{-O(1)}$. The same is true for $\varphi_0(i)$. We conclude that $\varphi(i)$ is conjugate to $\varphi_0(i)$ by a change of basis matrix $g' \in \GL_n(\C)$ with $\norm{g'} + \norm{\inv{g'}} \ll \rho^{-O(1)}$. Finally, thanks to Lemma~\ref{lm:realConj}, $g'$ can be chosen to be real. This finishes the proof for the case where $K \simeq \C$.

When $K$ is isomorphic to $\HH$, denote by $\psi \colon \HH \to \Mat_n(\R)$ the embedding of $\HH$ whose image is $K$. The homomorphism $\psi \colon \HH \to \Mat_n(\R)$ is uniquely determined by $\psi(i)$ and $\psi(j)$ and $E = \CC(\psi(i)) \cap \CC(\psi(j))$. As in the case where $K\simeq \C$, Lemma~\ref{lm:CentFix} gives the estimates $\norm{\varphi(i)},\norm{\varphi(j)}  \ll \rho^{-O(1)}$. 

The two matrices $\psi(i)$ and $\psi(j)$ satisfy 
\begin{equation}\label{eq:IetJ}
\psi(i)^2 = \psi(j)^2 = -\Id_n \text{ and }\psi(i)\psi(j) = -\psi(j)\psi(i).
\end{equation} 
Repeating the argument in the complex case, we see that $\psi(i)$ is conjugate to $\diag(i\Id_{\frac{n}{2}},-i\Id_{\frac{n}{2}})$ by a change of basis matrix satisfying the desired norm estimate. Write $\psi(i)$ and $\psi(j)$ in this new basis,  
\[ \psi(i) = \left[
\begin{array}{c|c}
i\Id_{\frac{n}{2}}  & 0 \\ \hline
 0 & -i\Id_{\frac{n}{2}}
\end{array}\right]\quad \text{and} \quad \psi(j) = \left[
\begin{array}{c|c}
A  & B \\ \hline
C & D
\end{array}\right].\]
The condition \eqref{eq:IetJ} is then equivalent to $A = D = 0$ and $BC = CB = -\Id_{\frac{n}{2}}$. It is easy to check that the conjugation by the block diagonal matrix $g' = \diag(C,\Id_{\frac{n}{2}})$ preserve $\psi(i)$ and conjugates $\psi(j)$ to 
\[\left[
\begin{array}{c|c}
 0 & -\Id_{\frac{n}{2}} \\ \hline
\Id_{\frac{n}{2}} & 0 
\end{array}\right].\]
Note that $\norm{g'} + \norm{g'^{-1}} = \norm{C} + \norm{B} \ll \norm{\phi(j)} \ll \rho^{-O(1)}$.  To conclude the proof we use Lemma~\ref{lm:realConj} to make sure the change of basis matrix is real.
\end{proof}

\section{Sum-product estimate in simple algebras}\label{sc:SumProduct}
We prove Theorem~\ref{thm:SumProduct} and Theorem~\ref{thm:SumProduct2} in this section.

\subsection{Sum-product theorem in \texorpdfstring{$\C^n$}{C^n}}
We will use Bourgain-Gamburd's sum-product theorem in $\C^n$ (Theorem~\ref{thm:etaDelta}) in a slightly stronger form.
\begin{coro} \label{coro:etaDelta}
Given $\kappa > 0$ and $n \geq 1$, there is $\alpha \geq 0$, $\beta > 0$ and a positive integer $s \geq 1$ such that, for $\epsilon > 0$ sufficiently small and $\delta > 0$ sufficiently small, the following holds. Let $A$ be a subset of $\Mat_n(\C)$. Assume that
\begin{enumerate}
\item $A \subset \Ball(0,\delta^{-\epsilon})$,
\item $\Ncov(A) \geq \delta^{-\kappa}$,
\item \label{it:AonDiag} $A \subset \Delta + \Ball(0,\delta^{1-\epsilon})$.
\end{enumerate}
Then there exists $\eta$ in the algebra generated by $A$ such that $\norm{\eta} = 1$ and
\begin{equation}\label{eq:etaDelta}
[0,\delta^\alpha] \eta \subset \sg{A}_s + \Ball(0,\delta^{\alpha + \beta}).
\end{equation}
\end{coro}
The corollary mainly says that the segment length $\delta^\alpha$ and the new scale $\delta^{\alpha+\beta}$ in Theorem~\ref{thm:etaDelta} can be chosen independently of $A$. This will be very useful when we use this diagonal case. Recall that $\Delta$ denotes the set of diagonal matrices in $\Mat_n(\C)$. In the following proof, $\pi_i(a)$ denotes the $i$-th diagonal entry of $a$ for $i \in \ensA{n}$ and $a \in \Mat_n(\C)$.

\begin{proof}
In this proof, let $\pi_i(a)$ denote the $i$-th diagonal entry of $a$ for any $i = 1, \dotsc, n$ and any $a \in \Mat_n(\C)$.
Partitioning $A$ into at most $\delta^{-O_n(\epsilon)}$ parts of diameter $1$ and choosing the part with the largest $\delta$\dash{}covering number, we see that $\Ncov(\Ball(0,1)\cap (A - A)) \geq \delta^{-\kappa + O_n(\epsilon)}$. Thus we can assume that $A \subset \Ball(0,1)$. By working at scale $\delta^{1-\epsilon}$, we can further assume $A \subset \Delta + \Ball(0,\delta)$. Theorem~\ref{thm:etaDelta} says that \eqref{eq:etaDelta} is true for some $\alpha$ and $\beta$ which may depend on $A$. Nevertheless, they can be bounded by constants depending only on $n$ and $\kappa$. What we need to show is that they can actually be chosen independently of $A$. This is evident for $\beta$ since \eqref{eq:etaDelta} gets only weaker when $\beta$ becomes smaller. Hence there exist $0 \leq \alpha_0 < C = C(n,\kappa)$, $s_0 = s_0(n,\kappa)\geq 1$ and $\eta \in \Delta$ such that $\norm{\eta} = 1$ and
\begin{equation}\label{eq:etaDelta0}
[0,\delta^{\alpha_0}] \eta \subset \sg{A}_{s_0} + \Ball(0,\delta^{\alpha_0 + \beta}).
\end{equation}
By replacing $\alpha_0$ with $\alpha_0 + \frac{1}{2}\beta$ and $\beta$ with $\frac{1}{2}\beta$ we can assume that $\alpha_0 \geq \beta$.
Fix an index $i$ such that $\abs{\pi_i(\eta)} \gg_n 1$.
Consider for any $s\geq 1$, the set
\begin{multline*}
\Omega_s = \bigl\{\omega \in \R_+ \mid \exists \xi \in \sg{A}_s,\, \norm{\xi} \leq s\delta^\omega,\, d(\xi,\Delta) \leq s\delta^{\omega + \frac{1}{2}\beta},\\
\abs{\pi_i(\xi)} \geq \frac{1}{s}\delta^\omega \text{ and } [0,1]\xi \subset \sg{A}_s + \Ball(0,s\delta^{\omega+\beta})\bigr\}.
\end{multline*}
We need to prove existence of $\alpha > 0$ and $s \geq 1$ depending only on $n$ and $\kappa$ such that $\alpha \in \Omega_s$. 

It follows from \eqref{eq:etaDelta0} that $\alpha_0 \in \Omega_{s_0}$. Now we show that if $\omega \in \Omega_s$ for some $s \geq 1$ then there is $s'=s'(s,n,\kappa) \geq 1$ such that $[\omega + \alpha_0, \omega + \alpha_0 + \frac{1}{2}\beta] \subset \Omega_{s'}$. Indeed, for any $\gamma \in [\alpha_0,\alpha_0 + \frac{1}{2}\beta]$, by \eqref{eq:etaDelta0}, there exists $a \in \sg{A}_{s_0}$ such that $\delta^\gamma \eta \in a + \Ball(0,\delta^{\alpha_0 + \beta})$. Thus $d(a,\Delta) \leq \delta^{\alpha_0 + \beta} \leq \delta^{\gamma + \frac{1}{2}\beta}$ and
\[\delta^\gamma \ll_n \pi_i(a) \leq \norm{a} \ll \delta^\gamma.\]
By multiplying $a$ to the relation $[0,1]\xi \subset \sg{A}_s + \Ball(0,s\delta^{\omega+\beta})$ we obtain
\[[0,1]\xi a \subset \sg{A}_{s+s_0} + \Ball(0,s\norm{a}\delta^{\omega+\beta}) \subset \sg{A}_{s+s_0} + \Ball(0,O(s)\delta^{\omega+\gamma+\beta}).\]
Moreover, $\norm{\xi a} \ll_n \norm{\xi} \norm{a} \ll s \delta^{\omega + \gamma}$ and
\[d(\xi a, \Delta) \ll_n d(\xi,\Delta)\norm{a} + \norm{\xi} d(a,\Delta) \ll s\delta^{\omega + \gamma + \frac{1}{2}\beta},\]
and for $\delta > 0$ sufficiently small,
\[\abs{\pi_i(\xi a)} \geq \abs{\pi_i(\xi)}\abs{\pi_i(a)} - O_n(d(\xi,\Delta)d(a,\Delta)) \gg_n \frac{1}{s}\delta^{\omega + \gamma} - O_n(\delta^{\omega + \gamma + \beta})\gg_n \frac{1}{s}\delta^{\omega + \gamma}\]
Hence $\omega + \gamma \in \Omega_{s'}$ for some $s'= s'(s,n,\kappa)$.

A simple induction yields that there exists a sequence $(s_k)_{k\geq 0}$ depending only on $n$ and $\kappa$ such that for all $k \geq 0$,
\[[(k+1)\alpha_0, (k+1)\alpha_0 + \frac{k}{2}\beta] \subset \Omega_{s_k}.\]

Recall that $\alpha_0$ depends on $A$ but it is bounded by $\beta \leq \alpha_0 \leq C$ where $\beta$ and $C$ are constants given by Theorem~\ref{thm:etaDelta} and depend only on $n$ and $\kappa$. Put $\alpha = \bigl(\ceil{\frac{2C}{\beta}} + 1\bigr)C$ and $K = \ceil{\frac{2\alpha}{\beta}}$. For any choice of $\alpha_0 \in [\beta,C]$, the equation
\[(k+1)\alpha_0\leq \alpha \leq (k+1)\alpha_0 + \frac{k}{2}\beta\]
has a solution $k$ satisfying $k \leq K$. It follows that $\alpha \in \Omega_{s_k} \subset \Omega_{s_K}$. This concludes the proof since $\alpha$ and $s_K$ depend only on $n$ and $\kappa$.
\end{proof}

\subsection{Small segment}
The key step in the proof of Theorem~\ref{thm:SumProduct} is to produce a small segment in $\sg{A}_s$.

\begin{prop}\label{pr:smallSeg}
Given a normed finite-dimensional simple algebra $E$ and $\kappa > 0$, there is $s \geq 1$ and $\epsilon > 0$, $\alpha,\beta > 0$ such that the following is true for $\delta > 0$ sufficiently small. Let $A$ be a subset of $E$. Assume that
\begin{enumerate}
\item $A \subset \Ball(0,\delta^{-\epsilon})$,
\item $\Ncov(A) \geq \delta^{-\kappa}$,
\item $A$ is $\delta^\epsilon$-away from subalgebras.
\end{enumerate}
Then there is $\eta \in E$, $\norm{\eta} = 1$ such that
\begin{equation}\label{eq:segmentEta}
[0,\delta^\alpha] \eta \subset \sg{A}_s + \Ball(0,\delta^{\alpha + \beta}).
\end{equation}
\end{prop}

\begin{proof}[Proof of Proposition~\ref{pr:smallSeg}]
First observe that we can assume without loss of generality that $E$ is a real subalgebra of $\Mat_n(\C)$ for some positive integer $n$ and it contains a least one element with $n$ distinct eigenvalues. Indeed, this is evident if $E$ is isomorphic to $\Mat_n(\C)$ or $\Mat_n(\R)$ since for the latter case we can embed naturally $\Mat_n(\R)$ in $\Mat_n(\C)$. We don't need to worry about the norm because all linear isomorphisms are bi-Lipschitz and bi-Lipschitz maps only change the constants in the assumptions and the conclusion of the proposition. If $E$ is isomorphic to $\Mat_n(\HH)$, then we can embed $\Mat_n(\HH)$ in $\Mat_{2n}(\C)$ by sending each entry $x+iy+jz+kw \in \HH$ to a $2\times 2$ block $\smatrixp{x + iy & z + iw \\ -z + iw & x - iy}$. It is easy to check that this embedding of $E$ contains a diagonal matrix with distinct diagonal entries.

In this proof, $s$ stands for an unspecified positive integer depending on $n$ and $\kappa$ that may increase from one line to another. Since $A$ is $\delta^\epsilon$-away from subalgebras, $\sg{A}_s$ is $\delta^{O(\epsilon)}$-away from linear subspaces by Proposition~\ref{pr:escapeV} applied to any $\eta \in A$ with $\norm{\eta} \geq \delta^{\epsilon}$. Therefore, without loss of generality, we can assume that $A$ is $\delta^\epsilon$-away from linear subspaces in $E$.

Consider $P \colon \Mat_n(\C) \to \C$ defined by
\begin{equation}\label{eq:sepEigen}
P(x) = \prod_{i < j}(\lambda_i -\lambda_j)^2
\end{equation}
where $\lambda_1, \dotsc, \lambda_n$ are eigenvalues of $x \in \Mat_n(\C)$ (the $n$ roots of the characteristic polynomial of $x$). The right-hand side of \eqref{eq:sepEigen} is symmetric in $(\lambda_i)$ and thus polynomial in the coefficients of the characteristic polynomial of $x$. Hence $x \mapsto \abs{P(x)}^2$ is a real polynomial on $\Mat_n(\C)$. Apply Lemma~\ref{lm:escapeSV} to $\abs{P}^2$ restricted to $E$. Since $E$ contains an element with $n$ distinct eigenvalues, we obtain an element $a \in \sg{A}_s$ such that $\abs{P(a)} > \delta^{O(\epsilon)}$ and consequently the eigenvalues $\lambda_1, \dotsc, \lambda_n$ of $a$ satisfy
\[\forall i\neq j, \quad \abs{\lambda_i - \lambda_j} > \delta^{O(\epsilon)}.\]

Now consider the map 
\begin{equation*}
\begin{matrix}
\phi \colon & \Mat_n(\C) &\to &\Mat_n(\C)\\
 &x &\mapsto & ax - xa.
\end{matrix}
\end{equation*}
Let $\kappa' = \frac{\kappa}{9n^2}$. We distinguish two cases according to the size of the image $\phi(A)$. First consider the case where
\[\Ncov\bigl(\phi(A)\bigr) \geq \delta^{2\kappa'} \Ncov(A).\]
In this case we can show a growth estimate. By Lemma~\ref{lm:traceSet}, there is a subset $A'\subset A^2$ such that $\tr_E(A')$ is $\delta$-separated and of size $\abs{\tr_E(A')} \geq \delta^{-4\kappa'}$. Here, we used the fact that $\dim(E) \leq 2n^2$. Observe that $\tr_E(\phi(x)) = 0$ for all $x \in E$, hence
\[\Ncov\bigl(\phi(A) + A'\bigr) \gg \Ncov(\phi(A)) \abs{\tr_E(A')} \geq \delta^{-2\kappa'} \Ncov(A).\]
Consequently,
\begin{equation}\label{eq:ACgrows}
\Ncov\bigl(\sg{A}_s\bigr) \geq \delta^{-\kappa'} \Ncov(A).
\end{equation}

Otherwise, we have
\[\Ncov\bigl(\phi(A)\bigr) < \delta^{2\kappa'} \Ncov(A),\]
then by cutting $A$ into "radius $\delta$" fibers, we see that
\[\Ncov(A) \ll \Ncov\bigl(\phi(A)\bigr) \max_{y} \Ncov\bigl(\phi^{-1}\bigl(\Ball(y,\delta)\bigr)\cap A\bigr).\]
Hence there is $y_* \in \Mat_n(\C)$ such that $\Ncov\bigl(\phi^{-1}\bigl(\Ball(y_*,\delta)\bigr) \cap A\bigr) \geq \delta^{-\kappa'}$. Put $A'' = \phi^{-1}\bigl(\Ball(y_*,\delta)\bigr) \cap A - \phi^{-1}\bigl(\Ball(y_*,\delta)\bigr) \cap A$ so that $\Ncov(A'') \geq \delta^{-\kappa'}$ and for all $x \in A''$, $\norm{\phi(x)} \ll \delta$. 

Recall that $\abs{\lambda_i - \lambda_j} > \delta^{O(\epsilon)}$ for all $i \neq j$. Hence we can apply Lemma~\ref{lm:effDiag} to the matrix $a$. We obtain a change of basis matrix $g \in \GL_n(\C)$ such that $\norm{g} + \norm{\inv{g}} < \delta^{-O(\epsilon)}$ and $ga\inv{g}$ is diagonal. Conjugation by $g$ will change any estimate only by a factor of $\delta^{-O(\epsilon)}$ or $\delta^{O(\epsilon)}$. That's why we can assume without loss of generality that $a = \diag(\lambda_1,\dotsc,\lambda_n)$. Then an explicit computation gives an expression for $\phi$ in the standard basis : if $x = (x_{ij})_{i,j} \in \Mat_n(\C)$, then
\[\phi(x) = \bigl( (\lambda_i -\lambda_j)x_{ij} \bigr)_{i,j}.\]
Therefore if $\norm{\phi(x)} \ll \delta$ then $d(x,\Delta) < \delta^{1 - O(\epsilon)}$. Hence $A'' \subset \Delta + \Ball(0,\delta^{1 - O(\epsilon)})$. Then Corollary~\ref{coro:etaDelta} gives constants $\alpha, \beta > 0$ and integer $s \geq 1$ depending only on $n$ and $\kappa'$ and a unit vector $\eta \in \Mat_n(\C)$ such that \eqref{eq:segmentEta} holds. 

What we have proved is that either the proposition holds or we have \eqref{eq:ACgrows}. If we are in the latter case, we can iterate the same argument to $\sg{A}_s$. After at most $O(\frac{n^2}{\kappa'})$ iterations, \eqref{eq:ACgrows} cannot be possible anymore, hence the proposition must be true.
\end{proof}
 
\subsection{Proof of Theorem~\ref{thm:SumProduct}.}

Once $\sg{A}_s$ contains a segment, it takes only a few more steps to produce a small ball.
\begin{prop}
\label{pr:smallSeg+}
Under the same assumptions, the conclusion of Proposition~\ref{pr:smallSeg} can be improved to
\begin{equation}\label{eq:smallBall}
\Ball(0,\delta^\alpha) \subset \sg{A}_s + \Ball(0,\delta^{\alpha + \beta}).
\end{equation}
\end{prop} 
\begin{proof}
Suppose that $\eta \in E$ is a unit vector satisfying \eqref{eq:segmentEta}. Since $A$ is $\delta^\epsilon$-away from subalgebras, by Proposition~\ref{pr:escapeV}, there is a $\delta^{O(\epsilon)}$-almost orthonormal basis of $E$ of the form $(a_i\eta b_i)_i$ with $a_i,b_i \in \sg{A}_{\dim(E)}$.
Then \eqref{eq:segmentEta} implies that for all $i = 1 ,\dotsc, \dim(E)$,
\[ [0,\delta^\alpha] a_i\eta b_i \subset \sg{A}_{s + 2\dim(E)} + \Ball(0,\delta^{\alpha + \beta - O(\epsilon)}).\]
Moreover, Lemma~\ref{lm:rhoOrth} yields
\[\Ball(0, \delta^{\alpha + O(\epsilon)}) \subset \sum_i [-\delta^\alpha,\delta^\alpha] a_i\eta b_i.\]
Hence \eqref{eq:smallBall} holds for sufficiently small $\epsilon$ and slightly worse $\alpha$, $\beta$ and $s$. 
\end{proof}

\begin{proof}[Proof of Theorem~\ref{thm:SumProduct}]
The idea is to apply Proposition~\ref{pr:smallSeg+} at various scales ranging from $\delta$ to $\delta^\epsilon$. Let $\epsilon_1,\alpha,\beta$ and $s$ be the constants given by Proposition~\ref{pr:smallSeg+}. Let $r = \ceil{\frac{\ln(\epsilon_0)}{\ln(\alpha) - \ln(\alpha + \beta)}}$ and for $k = 0,\dotsc, r$, define $\delta_k = \delta^{\frac{1}{\alpha}(\frac{\alpha}{\alpha + \beta})^k}$ so that $\delta_0^\alpha = \delta$, $\delta^{\epsilon_0} \leq \delta_r^\alpha$ and $\delta_k^{\alpha + \beta} = \delta_{k-1}^\alpha$ for all $k = 1, \dotsc, r$. 

We can require $\epsilon < \Inv{\alpha} (\frac{\alpha}{\alpha + \beta})^r \epsilon_1$ so that the assumptions of Proposition~\ref{pr:smallSeg} are satisfied at all the scales $\delta_k$, $k = 1, \dotsc, r$. Thus
\[\Ball(0,\delta_k^\alpha) \subset \sg{A}_s + \Ball(0,\delta_{k-1}^\alpha).\]
Hence,
\begin{align*}
\Ball(0, \delta^{\epsilon_0}) & \subset \Ball(0,\delta_r^\alpha) \\
& \subset \sg{A}_s + \Ball(0,\delta_{r-1}^\alpha) \\
& \subset \sg{A}_s + \sg{A}_s + \Ball(0,\delta_{r-2}^\alpha) \\
& \dotsb \\
& \subset \sg{A}_s + \dotsb + \sg{A}_s + \Ball(0,\delta_0^\alpha)
\end{align*}
Hence, $\Ball(0,\delta^{\epsilon_0}) \subset \sg{A}_{rs} + \Ball(0,\delta)$.
\end{proof}

\subsection{Proof of Theorem~\ref{thm:SumProduct2}}
We deduce Theorem~\ref{thm:SumProduct2} from Theorem~\ref{thm:SumProduct} and Lemma~\ref{lm:AplusAA}.
\begin{proof}[Proof of Theorem~\ref{thm:SumProduct2}]
Suppose that $A \subset E$ is a set satisfying the assumptions of Theorem~\ref{thm:SumProduct2} but 
\[\Ncov(A+A) + \Ncov(A + A\cdot A) \leq \delta^{-\epsilon}\Ncov(A).\]

On the one hand, applying Theorem~\ref{thm:SumProduct} with $\epsilon_0 = \frac{\dim(E) - \sigma}{2\dim(E)}$, we obtain an integer $s\geq 1$ such that
\[\Ball(0,\delta^{\epsilon_0}) \subset \sg{A}_s + \Ball(0,\delta).\]
Hence
\begin{equation}\label{eq:sgAbig}
\delta^{-\Inv{2}(\dim(E) + \sigma)} = \delta^{-(1-\epsilon_0)\dim(E)} \ll_E \Ncov(\sg{A}_s).
\end{equation}

On the other hand, by Lemma~\ref{lm:AplusAA},
\begin{equation}\label{eq:sgAsmall}
\Ncov(\sg{A}_s) \ll_E \delta^{-O_s(\epsilon)}\Ncov(A) \ll \delta^{-\sigma-O_s(\epsilon)}.
\end{equation}
The inequalities \eqref{eq:sgAbig} and \eqref{eq:sgAsmall} lead to a contradiction when $\epsilon$ is sufficiently small. 
\end{proof}

\begin{rmq}
Conversely, a growth statement like Theorem~\ref{thm:SumProduct2} always implies a statement like Theorem~\ref{thm:SumProduct}. The idea is to use the growth statement repeatedly until the set is nearly "full-dimensional" and then Fourier analysis shows that within a few more steps, it grows to "full dimension". For instance, \cite[Theorem 6]{Bourgain2010} deals with the one-dimensional case.
\end{rmq}

\section{Growth under linear action}\label{sc:ActionRn}
We prove Theorem~\ref{thm:ActionRn} in this section.

\subsection{Acting on \texorpdfstring{$\R^n$}{R^n}, probabilistic method}
We endow $\End(\R^n)$ with its usual operator norm. Recall that we denote by $\Mat_{\frac{n}{2}}(\C)_\R$ the standard embedding of $\Mat_{\frac{n}{2}}(\C)$ in $\End(\R^n)$ and by $\Mat_{\frac{n}{4}}(\HH)_\R$ that of $\Mat_{\frac{n}{4}}(\HH)$. First we study the special case where the collection of endomorphisms $A$ is the unit ball in $\End(\R^n)$ or one of these two subalgebras. Actually, the lemma below is a direct consequence of \cite[Proposition 1]{BourgainGamburd_SU}. Here, we present an elementary proof of this easier fact. 
\begin{lemm}\label{lm:BnotinS}
Given $\kappa > 0$ and $\sigma < n$, there is $\epsilon > 0$ such that the following holds for $\delta > 0$ sufficiently small. Let $E$ be $\End(\R^n)$ or $\Mat_{\frac{n}{2}}(\C)_\R$ or $\Mat_{\frac{n}{4}}(\HH)_\R$. Let $X$ be a subset of $\R^n$. Assume that
\begin{enumerate}
\item $X \subset \Ball(0,\delta^{-\epsilon})$,
\item $\forall \rho \geq \delta$, $\Ncov[\rho](X) \geq \delta^{\epsilon}\rho^{-\kappa}$,
\item $\Ncov(X) \leq \delta^{-\sigma-\epsilon}$,
\end{enumerate}
then 
\[\max_{f \in \Ball(0,1)\cap E} \Ncov(X + fX) \geq \delta^{-\epsilon} \Ncov(X).\]
\end{lemm}

\begin{proof}
Let $\mu$ be the normalized Lebesgue measure on $\Ball(0,1)\cap E$. It is easy to verify that $\mu$ satisfies the assumptions of the following proposition with $\tau = n$. Note that $\Ball(0,1)\cap E$ contains the identity $\Id$.
\end{proof}

\begin{prop}\label{pr:spread}
Given $\kappa > 0$ and $0 < \sigma < \tau$, there is $\epsilon > 0$ such that the following holds for $\delta > 0$ sufficiently small. Let $X$ be a subset of $\R^n$ and $\mu$ a probability measure on $\End(\R^n)$. If
\begin{enumerate}
\item $X \subset \Ball(0,\delta^{-\epsilon})$,
\item $\forall \rho \geq \delta$, $\Ncov[\rho](X) \geq \delta^{\epsilon}\rho^{-\kappa}$,
\item \label{it:Xnotall} $\Ncov(X) \leq \delta^{-\sigma-\epsilon}$,
\item The support of $\mu$, $\Supp(\mu) \subset \Ball(0,\delta^{-\epsilon})$,
\item For all $\rho \geq \delta$ and all $v,w \in \R^n$ with $\norm{v}=1$, 
\[\mu \bigl(\ens{f \in \End(\R^n) \mid fv \in w + \Ball(0, \rho)}\bigr) \leq \delta^{-\epsilon}\rho^\tau,\]
\end{enumerate}
then 
\[\Ncov(X + X) + \max_{f \in \Supp(\mu)}\Ncov(X + fX) \geq \delta^{-\epsilon} \Ncov(X).\]
\end{prop}

\begin{proof}
Let $X$ and $\mu$ be as in the statement. Assume that $\Ncov(X+X) \leq \delta^{-\epsilon}\Ncov(X)$. For all $\rho \geq \delta$ we have
\[\Ncov(X+X) \gg \Ncov[\rho](X) \max_{w\in\R^n} \Ncov(X\cap \Ball(w,\rho)).\]
Therefore, by assumption (ii),
\begin{equation}\label{eq:NonCrA}
\max_w \Ncov(X\cap \Ball(w,\rho)) \leq \delta^{-O(\epsilon)}\rho^\kappa \Ncov(X).
\end{equation}

Let $f$ be a random variable following the law $\mu$. Define $\phi_f \colon \R^n \times \R^n \to \R^n $ by
\[\phi_f(x,y) = x + fy.\]
This map is $\delta^{-\epsilon}$-Lipschitz by assumption (iv). Consider the $\phi_f$-energy of $X \times X$. By Lemma~\ref{lm:phiEnergy}(i), 
\[ \Ncov(X + fX) \gg \frac{\Ncov(X)^4}{\En_\delta(\phi_f, X \times X)}.\]
Hence by Jensen's inequality,
\[\Espr{\Ncov(X + fX)} \gg \frac{\Ncov(X)^4}{\Espr{\En_\delta(\phi_f, X \times X)}}.\]

The rest of the proof consists of bounding the expectation $\Espr{\En_\delta(\phi_f, X \times X)}$ from above. Fix $\tilde X$ a maximal $\delta$-separated subset of $X$. By Lemma~\ref{lm:phiEnergy}(ii),
\begin{equation*}
\Espr{\En_\delta(\phi_f, X \times X)} \ll \sum_{x,x',y,y'\in \tilde X} \Prob{f(y - y') \in x' - x + B(0,\delta^{1-2\epsilon})}.
\end{equation*}

Let $\rho \geq \delta$ be a scale to be chosen later. We split the sum into two parts according to whether $\norm{y - y'} \geq \rho$. If it is the case then the assumption (v) yields for any $x, x' \in \tilde X$,
\[\Prob{f(y - y') \in x' - x + B(0,\delta^{1-2\epsilon})} \leq \delta^{-O(\epsilon)}\rho^{-\tau}\delta^\tau.\]

To estimate the other part we note that the number of pairs $(y,y')$ such that $\norm{y-y'}\leq \rho$ can be bounded using \eqref{eq:NonCrA}.
\[\# \{(y,y') \in \tilde X \times \tilde X \mid \norm{y-y'}\leq \rho\} \leq \delta^{-O(\epsilon)}\rho^\kappa\Ncov(X)^2.\]
Moreover, for any $x,y,y' \in \tilde X$,
\[\sum_{x' \in \tilde X}\Prob{f(y - y') \in x' - x + B(0,\delta^{1-2\epsilon})} \leq \delta^{-O(\epsilon)}\]
since the events on the left-hand side can occur simultaneously for at most $\delta^{-O(\epsilon)}$ different $x' \in \tilde X$.  

By combining these inequalities and assumption (iii) and taking $\rho = \delta^{\frac{\tau - \sigma}{\tau + \kappa}}$, we obtain
\begin{align*}
\Espr{\En_\delta(\phi_f, X \times X)} &\leq \delta^{-O(\epsilon)}\bigl(\rho^{-\tau}\delta^\tau\abs{\tilde X}^4 + \rho^\kappa\abs{\tilde X}^3\bigr)\\
&\leq \delta^{-O(\epsilon)} (\delta^{\tau - \sigma}\rho^{-\tau} + \rho^\kappa)\Ncov(X)^3\\
&\leq \delta^{\frac{\kappa(\tau - \sigma)}{\tau + \kappa} - O(\epsilon)} \Ncov(X)^3.
\end{align*}
It follows that when $\epsilon$ is small enough, $\Espr{\Ncov(X + fX)} \geq \delta^{-\epsilon}\Ncov(X)$.
\end{proof}

\subsection{Acting on \texorpdfstring{$\R^n$}{R^n}, using sum and product}
Before using Theorem~\ref{thm:SumProduct} we need to know how matrix addition and multiplication affects the estimate~\eqref{eq:ActionRn}. In this subsection we show that in order to establish \eqref{eq:ActionRn}, it suffices to prove it with $\sg{A}_s + \Ball(0,\delta)$ in place of $A$ for some $s\geq 1$. 

Let $\epsilon > 0$. Let $X \subset \R^n$ be a bounded subset. Similarly to the consideration of good elements in \cite[Proposition 3.3]{BourgainKatzTao} and the basic construction in \cite[Proposition 3.1]{Tao_Ring}, we define at scale $\delta > 0$,
\[\Sset(X;\delta^{-\epsilon}) = \ens{f \in \Ball(0,\delta^{-\epsilon}) \cap \End(\R^n) \mid \Ncov(X + fX) \leq \delta^{-\epsilon}\Ncov(X)}.\]

The set $\Sset(X;\delta^{-\epsilon})$ exhibits the structure of an "approximate ring".
\begin{lemm}\label{lm:setSA}
Let $X \subset \Ball(0,\delta^{-\epsilon})$ be a subset of $\R^n$, we have
\begin{enumerate}
\item If $a \in \Sset(X;\delta^{-\epsilon})$ and $b \in \End(\R^n)$ such that $\norm{a-b} \leq \delta^{1-\epsilon}$, then $b \in \Sset(X;\delta^{-O(\epsilon)})$
\item If $\Id, a, b \in \Sset(X;\delta^{-\epsilon})$, then $a + b$, $a - b$ and $ab$ all belong to $\Sset(X;\delta^{-O(\epsilon)})$.
\item Suppose that $a$ invertible and $\norm{\inv{a}} \leq \delta^{-\epsilon}$. If $a \in \Sset(X;\delta^{-\epsilon})$, then $\inv{a} \in \Sset(X;\delta^{-O(\epsilon)})$.
\item If $\Id, a_1, \dotsc, a_s \in \Sset(X;\delta^{-\epsilon})$, then 
\[\Ncov(X + a_1 X + \dotsb + a_s X) \leq \delta^{-O(s\epsilon)} \Ncov(X).\]
\item If $\Id, a \in \Sset(X;\delta^{-\epsilon})$, then for all $\rho \geq \delta$, we have 
\[\Ncov[\rho](X+aX) \leq \delta^{-O(\epsilon)}\Ncov[\rho](X).\]
In other words, $a \in \Sset[\rho](X;\delta^{-O(\epsilon)})$.
\end{enumerate}
\end{lemm}

\begin{proof}
\begin{enumerate}
\item If $\Ncov(X + aX) \leq \delta^{-\epsilon}\Ncov(X)$ and $\norm{a-b} \leq \delta^{1-\epsilon}$, then
\[X + bX \subset X + aX + \Ball(0,\delta^{1-2\epsilon}).\]
Hence $\Ncov(X + bX) \leq \delta^{-O(\epsilon)}\Ncov(X)$.
\item Let $a,b \in \Sset(X;\delta^{-\epsilon})$. By Ruzsa's covering lemma (Lemma~\ref{lm:RuzsaCov}),
\[aX \subset X - X + \OB(\delta^{-\epsilon}) + \Ball(0, \delta),\]
and
\[bX \subset X - X + \OB(\delta^{-\epsilon}) + \Ball(0, \delta).\]
Hence,
\[X + (a+b)X \subset 3X-2X + \OB(\delta^{-2\epsilon}) + \Ball(0,2\delta),\]
\[X + (a-b)X \subset 3X-2X + \OB(\delta^{-2\epsilon}) + \Ball(0,2\delta).\]
And finally by the Plünnecke-Ruzsa inequality (Lemma~\ref{lm:RuzsaSum}), 
\[\Ncov(X + (a+b)X),\; \Ncov(X + (a-b)X) \leq \delta^{-O(\epsilon)}\Ncov(X).\]

Moreover, since $\norm{a} \leq \delta^{-\epsilon}$,
\begin{align*}
X + abX &\subset X + a(X-X + \OB(\delta^{-\epsilon}) + \Ball(0,\delta))\\
&\subset X + aX - aX + \OB(\delta^{-\epsilon}) + \Ball(0,\delta^{1-\epsilon})\\
&\subset 3X - 2X + \OB(\delta^{-3\epsilon}) + \Ball(0,3\delta^{1-\epsilon})
\end{align*}
Hence, $\Ncov(X + abX) \leq \delta^{-O(\epsilon)}\Ncov(X)$.
\item If $a \in \GL_n(\R)$ and $\norm{\inv{a}} \leq \delta^{-\epsilon}$, then $X + \inv{a}X = \inv{a}(X + aX)$ and hence $\Ncov(X + \inv{a}X) \leq \delta^{-O(\epsilon)}\Ncov(X + aX)$.
\item The argument is similar as in (ii).
\item For all $\rho \geq \delta$ we have
\[\Ncov(X) \leq \max_{x \in \R^n} \Ncov(X \cap\Ball(x,\rho)) \Ncov[\rho](X)\]
and for any $x \in \R^n$,
\[\Ncov(X + X + aX) \gg \Ncov(X\cap \Ball(x,\rho)) \Ncov[\rho](X + aX).\]
If $\Id, a \in \Sset(X;\delta^{-\epsilon})$, then by (iv),
\[\Ncov(X + X + aX) \leq \delta^{-O(\epsilon)} \Ncov(X).\]
We obtain the desired estimate by combining the three inequalities above.\qedhere
\end{enumerate}
\end{proof}

\subsection{Almost-generating a subalgebra}
Now we can focus on the growth of $A$ as a set of matrices. We cannot use Theorem~\ref{thm:SumProduct} yet since we do not know if $A$ is away from subalgebras. In this subsection, we show that we can find a subalgebra $E_0$ of $\End(\R^n)$ such that, at some scale, $A$ is contained in $E_0$ and is effectively away from subalgebras in $E_0$. This subalgebra $E_0$ can be viewed as approximately generated by $A$. Moreover, under the quantitative irreducibility condition, $E_0$ shall be described by Proposition~\ref{pr:Wedder}.

\begin{prop}\label{pr:gAgE}
Given $\epsilon_1 > 0$, there is $c > 0$ such that for any $0 < \epsilon < c$, the following holds for all $\delta > 0$ sufficiently small. Let $A$ be a subset of $\End(\R^n)$. If 
\begin{itemize}
\item $A \subset \Ball(0,\delta^{-\epsilon})$,
\item $A$ acts $\delta^\epsilon$-irreducibly on $\R^n$;
\end{itemize}
then there exists $\delta_1 \in [\delta, \delta^c]$ and $g \in \GL_n(\R)$ with $\norm{g} + \norm{\inv{g}} \leq \delta^{-O(\epsilon)}$ such that for $E = \End(\R^n)$, $\Mat_{\frac{n}{2}}(\C)_\R$ or $\Mat_{\frac{n}{4}}(\HH)_\R$,
\[gA\inv{g} \subset E + \Ball(0,\delta_1)\]
and for all proper subalgebras $F$ of $E$, 
\[\exists a \in A, d(ga\inv{g},F) > \delta_1^{\epsilon_1}.\]
\end{prop}

\begin{proof}
Let $l_0$ be the largest among all integers $l \in \N$ such that there exists a subalgebra $E$ of $\End(\R^n)$ of codimension $l$ and such that $A \subset E + \Ball(0,\delta^{(\frac{\epsilon_1}{4})^l})$. We know $l_0$ exists since $0$ is clearly such an $l$. Set $\delta_1 = \delta^{\frac{1}{2}(\frac{\epsilon_1}{4})^{l_0}}$. Thus $\delta \leq \delta_1 \leq \delta^{c_1}$ with $c_1 = \frac{1}{2}(\frac{\epsilon_1}{4})^{n^2}$. By the definition of $l_0$ there is a subalgebra $E_0$ of $\End(\R^n)$ such that
\begin{equation}\label{eq:AgenE0}
A \subset E_0 + \Ball(0,\delta_1^2)
\end{equation}
and for any proper subalgebra $F$ of $E_0$, there is $a \in A$ such that 
\begin{equation}\label{eq:daFlarge} 
d(a,F) > \delta_1^{\frac{\epsilon_1}{2}}
\end{equation}

We shall apply Proposition~\ref{pr:Wedder} to this subalgebra $E_0$ in order to conjugate it into one of the three "model" subalgebras. For any nonzero linear subspace $W$ of $\R^n$, since $A$ acts $\delta^\epsilon$-irreducibly , there is $w \in W \cap \Ball(0,1)$ and $a \in A$ such that $d(aw,W) > \delta^\epsilon$. Then there is $a' \in E_0$ such that $\norm{a-a'} < \delta_1$. Hence $\norm{a'} < \delta^{-\epsilon} + \delta_1 < \delta^{-2\epsilon}$ and 
\[d(\frac{a'}{\norm{a'}}w,W) > \delta^{2\epsilon}(\delta^{\epsilon} - \delta_1) \geq \delta^{2\epsilon}(\delta^{\epsilon} - \delta^{c_1}) > \delta^{4\epsilon}\]
provided $\epsilon \leq \frac{c_1}{2}$. Thus $E_0$ acts $\delta^{4\epsilon}$-irreducibly on $\R^n$. We conclude that there is $g \in \GL_n(\R)$ with $\norm{g} + \norm{\inv{g}} < \delta^{-O(\epsilon)}$ and $E = gE_0\inv{g}$ is one of these three subalgebras : $\End(\R^n)$, $\Mat_{\frac{n}{2}}(\C)_\R$, $\Mat_{\frac{n}{4}}(\HH)_\R$.

The map $x \mapsto gx\inv{g}$ is $\delta^{-O(\epsilon)}$-bi-Lipschitz. Hence by \eqref{eq:AgenE0} and \eqref{eq:daFlarge},
\[gA\inv{g} \subset E + \Ball(0,\delta^{-O(\epsilon)}\delta_1^2)\]
and $gA\inv{g}$ is $\delta^{O(\epsilon)}\delta_1^{\frac{\epsilon_1}{2}}$-away from proper subalgebras of $E$. When $\epsilon$ is small enough compared to $c_1\epsilon_1$, we have both 
\[\delta^{-O(\epsilon)}\delta_1^2 \leq \delta_1 \quad \text{and} \quad \delta^{O(\epsilon)}\delta_1^{\frac{\epsilon_1}{2}} \geq \delta_1^{\epsilon_1}.\]
And this completes the proof.
\end{proof}

\begin{rmq}
From the proof we see that the new scale $\delta_1$ can be chosen such that $\delta$ is an integer power of $\delta_1$.
\end{rmq}

\subsection{Proof of Theorem~\ref{thm:ActionRn}.} The proof consists of putting together what precedes. The only technical difficulty is due to the change of the working scale required by Proposition~\ref{pr:gAgE}. 

\begin{proof}[Proof of Theorem~\ref{thm:ActionRn}]
Assume for a contradiction that $A\cup \ens{\Id} \subset \Sset(X;\delta^{-\epsilon})$.

Let $\epsilon_1 > 0$ be a small constant to be chosen later. By applying Proposition~\ref{pr:gAgE} to the set $A$, we get a constant $c>0$ depending on $\epsilon_1$, a new scale $\delta_1 \in [\delta, \delta^c]$, an isomorphism $g \in \GL_n(\R)$ and a subalgebra $E = \End(\R^n)$ or $\Mat_{\frac{n}{2}}(\C)_\R$ or $\Mat_{\frac{n}{4}}(\HH)_\R$ such that $\norm{g} + \norm{\inv{g}} \leq \delta^{-O(\epsilon)}$ and
\[gA\inv{g} \subset E + \Ball(0,\delta_1)\]
and the projection of $gA\inv{g}$ on $E$, which we will denote by $A'$, is $\delta_1^{\epsilon_1}$-away from subalgebras in $E$. Moreover, we can verify the other assumptions of Theorem~\ref{thm:SumProduct} for the scale $\delta_1$ and parameter $\epsilon_1$ : we have
\begin{equation}\label{eq:BandB}
A' \subset gA\inv{g} + \Ball(0,\delta_1).
\end{equation}
Thus
\[A' \subset \Ball(0,\delta^{-O(\epsilon)} + \delta_1) \subset \Ball(0,\delta_1^{-O(\frac{\epsilon}{c})}).\]
And for all $\rho \geq \delta_1$, since $gA\inv{g} \subset A' + \Ball(0,\rho)$, we have
\[\Ncov[\rho](A') \gg \Ncov[\rho](gA\inv{g}) \gg \delta^{O(\epsilon)}\Ncov[\rho](A) \geq \delta^{O(\epsilon)}\rho^{-\kappa} \geq \delta_1^{O(\frac{\epsilon}{c})}\rho^{-\kappa}.\]
So when $\epsilon$ is sufficiently small depending on $\epsilon_1$ and $c$, we have both $A' \subset \Ball(0,\delta_1^{-\epsilon_1})$ and $\forall \rho \geq \delta_1$, $\Ncov[\rho](A')\geq \delta_1^{\epsilon_1} \rho^{-\kappa}$.

Let $\epsilon_0 > 0$ be a small constant to be chosen later in the proof. Applying Theorem~\ref{thm:SumProduct} to the set $A'$ at scale $\delta_1$ inside the algebra $E$, we obtain an integer $s \geq 1$ depending only on $n$, $\kappa$ and $\epsilon_0$ such that
\begin{equation}\label{eq:EinBC}
E \cap\Ball(0,\delta_1^{\epsilon_0}) \subset \sg{A'}_s + \Ball(0,\delta_1),
\end{equation}
if $\epsilon_1$ is chosen small enough depending on $n$, $\kappa$ and $\epsilon_0$.

From $A \subset \Sset(X;\delta^{-\epsilon})$, we get for all $a \in A$,
\[\Ncov(gX + ga\inv{g}gX) = \Ncov\bigl(g(X + aX)\bigr) \leq \delta^{-O(\epsilon)}\Ncov(X) \leq \delta^{-O(\epsilon)}\Ncov(gX).\]
In other words, $gA\inv{g} \subset \Sset(gX,\delta^{-O(\epsilon)})$. From Lemma~\ref{lm:setSA}(v), it follows that $gA\inv{g} \subset \Sset[\delta_1](gX;\delta^{-O(\epsilon)})$. Then by Lemma~\ref{lm:setSA}(i), 
\[A' \subset \Sset[\delta_1](gX;\delta_1^{-O(\frac{\epsilon}{c})}) \subset \Sset[\delta_1](gX;\delta_1^{-\epsilon_1}).\]
Repeated use of Lemma~\ref{lm:setSA}(ii) yields
\[\sg{A'}_s \subset \Sset[\delta_1](gX;\delta_1^{-O_s(\epsilon_1)}).\]
Then by \eqref{eq:EinBC} and Lemma~\ref{lm:setSA}(i),
\[E\cap \Ball(0,\delta_1^{\epsilon_0}) \subset \Sset[\delta_1](gX;\delta_1^{-O_s(\epsilon_1)}).\]
In particular, $\delta_1^{\epsilon_0}\Id \in \Sset[\delta_1](gX;\delta_1^{-O_s(\epsilon_1)})$. Hence, by Lemma~\ref{lm:setSA}(iii),
\[\delta_1^{-\epsilon_0}\Id \in \Sset[\delta_1](gX;\delta_1^{-O_s(\epsilon_1) - O(\epsilon_0)}).\]
And again by Lemma~\ref{lm:setSA}(ii), $E\cap \Ball(0,1) \subset  \Sset[\delta_1](gX;\delta_1^{-O_s(\epsilon_1) - O(\epsilon_0)})$.

Hence for any given $\epsilon_2 > 0$, we can choose sufficiently small $\epsilon_0 > 0$ and $\epsilon_1 > 0$ accordingly so that 
\begin{equation}\label{eq:EinSA}
E\cap \Ball(0,1) \subset  \Sset[\delta_1](gX;\delta_1^{-\epsilon_2}).
\end{equation}
Take $\epsilon_2$ to be the constant given by Lemma~\ref{lm:BnotinS} depending on $\sigma$ and $\kappa$. We would like to apply Lemma~\ref{lm:BnotinS} to the set $gX$ at scale $\delta_1$. It's easy to see that when $\epsilon$ is small enough,
\[ gX \subset \Ball(0,\delta^{-O(\epsilon)}) \subset \Ball(0,\delta_1^{-\epsilon_2}),\]
and for all $\rho \geq \delta_1$,
\[\Ncov[\rho](gX) \geq \delta^{O(\epsilon)}\Ncov(X) \geq \delta^{O(\epsilon)}\rho^{-\kappa} \geq \delta_1^{\epsilon_2}\rho^{-\kappa}.\]
So the first two assumptions in Lemma~\ref{lm:BnotinS} are satisfied but the conclusion fails by \eqref{eq:EinSA}. This means the assumption \ref{it:Xnotall} must fail, namely, $\Ncov[\delta_1](gX) > \delta_1^{-\sigma-\epsilon_2}$. Therefore,
\begin{equation}\label{eq:AalmostB}
\Ncov[\delta_1](X) \geq \delta^{O(\epsilon)}\delta_1^{-\sigma-\epsilon_2}.
\end{equation}

In other words, at scale $\delta_1$, the set $X$ is almost full in $\Ball(0,\delta^{-\epsilon})$. The idea of the rest of the proof is to use multiplication of elements of $A$ to propagate estimate~\eqref{eq:AalmostB} to smaller scales until we reach the original scale $\delta$ where we have the assumption $\Ncov(X) \leq \delta^{-\sigma-\epsilon}$.

We have by \eqref{eq:EinBC} and \eqref{eq:BandB}, 
\[\delta_1^{\Inv{2}}\Id \in \sg{A'}_s + \Ball(0,\delta_1) \subset g\sg{A}_s\inv{g} + \Ball(0,\delta^{-O_s(\epsilon)}\delta_1).\] 
Hence there exists $a \in \sg{A}_s$ such that $a \in \delta_1^{\Inv{2}} \Id + \Ball(0,\delta^{-O_s(\epsilon)}\delta_1)$. Taking the square, we obtain $a^2 \in \delta_1\Id + \Ball(0,\delta^{-O_s(\epsilon)}\delta_1^{\frac{3}{2}})$. Thus, when $\epsilon$ is sufficiently small,
\[a^2 \in \delta_1\Id + \Ball(0,\frac{\delta_1}{2}).\]
For all $\rho > 0$, the multiplication by $a^2$ will transform a $\rho$-separated set in $\R^n$ into a $\frac{\delta_1\rho}{2}$-separated set. Hence,
\[\Ncov[\delta_1\rho](a^2X) \gg \Ncov[\rho](X).\]
Moreover $a^2X \subset \Ball(0,\delta^{-2\epsilon}\delta_1)$. So
\[\Ncov[\delta_1\rho](X + a^2X) \geq \delta^{O(\epsilon)} \Ncov[\delta^{-2\epsilon}\delta_1](X) \Ncov[\delta_1\rho](a^2X) \geq \delta^{O(\epsilon)}\Ncov[\delta_1](X)\Ncov[\rho](X).\] 
Lemma~\ref{lm:setSA} tells us $a^2 \in \Sset(X,\delta^{-O_s(\epsilon)})$ and for all $\rho \geq \delta_1^{-1}\delta$,
\[\Ncov[\delta_1\rho](X+a^2X) \leq \delta^{-O_s(\epsilon)} \Ncov[\delta_1\rho](X).\]
Combining the two estimates above, we obtain
\begin{equation}\label{eq:NcovdrA}
\Ncov[\delta_1\rho](X) \geq \delta^{O_s(\epsilon)}\Ncov[\delta_1](X)\Ncov[\rho](X)
\end{equation}
whenever $\delta_1\rho \geq \delta$.

According to the remark after the proof of Proposition~\ref{pr:gAgE}, we can assume that there is an integer $k$ with $1 \leq k \leq \Inv{c}$ such that $\delta = \delta_1^k$. Applying \eqref{eq:NcovdrA} to $\rho = \delta_1^i$, $i = 1,\dotsc, k-1$, we obtain
\[\Ncov[\delta_1^k](X) \geq \delta^{O_s(k\epsilon)} \Ncov[\delta_1](X)^k.\] 
Hence, by \eqref{eq:AalmostB} and the assumption on $X$,
\[\delta^{-\sigma -\epsilon} \geq \Ncov(X) \geq \delta^{-\sigma-\epsilon_2 + O_s(k\epsilon)},\]
which is clearly impossible when $\epsilon$ is small enough and this finishes the proof. 
\end{proof}

\begin{rmq}
Theorem~\ref{thm:ActionRn} together with Lemma~\ref{lm:awWwaW} yields another proof of Theorem~\ref{thm:SumProduct2}. It suffices to consider the action of left and right multiplications of elements of $A$ on the set $A$. Actually, this proves something slightly stronger. Namely, the conclusion of Theorem~\ref{thm:SumProduct2} can be improved to
\[\Ncov(A + A) + \max_{a \in A}\Ncov(A + a\cdot A) + \max_{a \in A}\Ncov(A + A\cdot a)  \geq \delta^{-\epsilon} \Ncov(A).\]
\end{rmq}

\section{A sum-product estimate in simple Lie algebras}\label{sc:Lie}

In this last section, we prove Corollary~\ref{cr:SumBracket}. Let $\gf$ be a normed finite-dimensional real simple Lie algebra. We consider the adjoint representation $\ad \colon \gf \to \End(\gf)$, i.e. $\ad(a)x = [a,x]$ for all $a,x \in \gf$.

\begin{proof}[Proof of Corollary~\ref{cr:SumBracket}]
It suffices to apply Theorem~\ref{thm:ActionRn} to the set $A \subset \gf$ and the set of endomorphisms $\ad(A) \subset \End(\gf)$. Note that the kernel of $\ad$ is the center of $\gf$ which is trivial for $\gf$ is simple. Therefore, $\ad$ is a bi-Lipschitz map from $\gf$ to its image. This gives the non-concentration condition on $\ad(A)$. And the quantitative irreducibility condition of Theorem~\ref{thm:ActionRn} is guaranteed by the following lemma.
\end{proof}

\begin{lemm}
Let $0 < \rho \leq \Inv{2}$ be a parameter. Let $A$ be a subset in a normed simple Lie algebra $\gf$ of finite dimension. Assume that $A \subset \Ball(0,\inv{\rho})$ and that $A$ is $\rho$-away from Lie subalgebras in $\gf$. Then $\ad(A)$ acts $\rho^{O_{\gf}(1)}$-irreducibly on $\gf$.
\end{lemm}

\begin{proof}
Without loss of generality, we can assume the norm on $\gf$ to be Euclidean. The statement and proof of Lemma~\ref{lm:FiniteA} remains valid when the word "subalgebra" is replaced by "Lie subalgebra". Actually this fact is implicit in \cite[Lemma 2.5]{Saxce}. Therefore, as in the proof of Lemma~\ref{lm:awWwaW}, we can assume that $A$ is finite of cardinality at most $n$ and contained in $\Ball(0,1)$. Write $A = \ens{a_1,\dotsc,a_n}$.

Suppose the conclusion does not hold, which means there is a linear subspace $W_0$ of dimension $0 < k < n$ such that for all $w \in W_0 \cap \Ball(0,1)$ and all $i = 1,\dotsc,n$, $d(\ad(a_i)w,W_0) < \rho^C$ where $C$ is a large constant to be determined by the use of the \L{}ojasiewicz inequality below. Consider the following real analytic map :
\begin{equation*}
\raisebox{1.2ex}{$\displaystyle f \colon$} \begin{matrix}
& \Gr(k,\gf) \times \gf^n &\to &\R\\
&(W;x_1,\dotsc,x_n) &\mapsto & \displaystyle\sum_{i=1}^n\int_{\raisebox{-0.4ex}{$\scriptstyle W\cap \Ball(0,1)$}} \hspace{-2em} d(\ad(x_i)w,W)^2 \,\dd w
\end{matrix}
\end{equation*}
From the above we have $f(W_0;a_1,\dotsc,a_n) \ll \rho^C$. Application of the \L{}ojasiewicz inequality (Theorem~\ref{thm:Lojasiewicz}) to the compact set $\Gr(k,\gf) \times \Ball(0,1)^n$ gives $W_1 \in \Gr(k,\gf)$ and $b_1,\dotsc,b_n \in \gf$ such that $f(W_1;b_1,\dotsc,b_n) = 0$ and $\forall i = 1, \dotsc, n$, $\norm{a_i - b_i} < \rho$ when the constant $C$ is chosen large enough. The fact $f(W_1;b_1,\dotsc,b_n) = 0$ is equivalent to every $b_i$ being in the Lie subalgebra
\[\gf_{W_1} = \ens{ x\in \gf \mid \ad(x)W_1 \subset W_1}.\]

Now our set $A$ is not $\rho$-away from the Lie subalgebra $\gf_{W_1}$. Hence $\gf_{W_1}$ must be $\gf$, which in turn implies that $W_1$ is an ideal in $\gf$. This contradicts the simplicity of $\gf$.
\end{proof}

\bibliographystyle{abbrv} 
\bibliography{here}

\end{document}